\documentclass[11pt,twoside,]{amsart}
\usepackage{amsmath, amsthm, amscd, amsfonts, amssymb, color}
\usepackage[dvips]{graphicx}
\usepackage{tikz}
\usepackage[bookmarksnumbered, plainpages]{hyperref}
\newtheorem{thm}{Theorem}[section]
\newtheorem{cor}[thm]{Corollary}
\newtheorem{lem}[thm]{Lemma}
\newtheorem{prop}[thm]{Proposition}
\newtheorem{defn}[thm]{Definition}
\newtheorem{rem}[thm]{\bf{Remark}}

\numberwithin{equation}{section}
\def\pn{\par\noindent}

\begin{document}

\title{Graphs cospectral with a friendship graph or its complement}
\author{Alireza Abdollahi$^*$  Shahrooz Janbaz and Mohammad Reza Oboudi}

\thanks{{\scriptsize
\hskip -0.4 true cm MSC(2010): Primary: 05C15
\newline Keywords: Friendship graphs, cospectral graphs, adjacency eigenvalues.\\
$*$Corresponding author}}
\maketitle

\begin{abstract}
Let $n$ be any positive integer and  let $F_n$ be the friendship (or Dutch windmill) graph with $2n+1$ vertices and $3n$ edges. Here we study graphs with the same adjacency spectrum as the $F_n$. Two graphs are called cospectral if the eigenvalues multiset  of their adjacency matrices are the same. Let $G$ be a graph cospectral with $F_n$. Here we prove that if $G$ has no cycle of length $4$ or $5$, then $G\cong F_n$. Moreover if $G$ is connected and planar then $G\cong F_n$.
All but one of connected components of $G$ are isomorphic to $K_2$.
The complement $\overline{F_n}$ of the friendship graph is determined by its adjacency eigenvalues, that is, if $\overline{F_n}$ is cospectral with a graph $H$, then $H\cong \overline{F_n}$.
\end{abstract}

\vskip 0.2 true cm


\pagestyle{myheadings}
\markboth{\rightline {\scriptsize  Abdollahi, Janbaz and Oboudi}}
         {\leftline{\scriptsize Graphs cospectral with a friendship graph}}

\bigskip
\bigskip

\section{\bf Introduction}
\vskip 0.4 true cm

All graphs in this paper are simple of finite orders, i.e., graphs are undirected with no loops or parallel edges and with finite number of vertices. Let $V(G)$ and $E(G)$ denote the vertex set and edge set of a graph $G$, respectively. Also, $A(G)$ denotes the $(0,1)$-adjacency matrix of graph $G$. The characteristic polynomial of $G$ is det$\left(\lambda I-A(G)\right)$, and we denote it by $P_{G}(\lambda)$. The roots of $P_{G}(\lambda)$ are called the adjacency eigenvalues of $G$ and since $A(G)$ is real and symmetric, the eigenvalues are real numbers. If $G$ has $n$ vertices, then it has $n$ eigenvalues and we denote its eigenvalues in descending order as $\lambda_1\geq \lambda_2\geq \cdots \geq \lambda_n$. Let $\lambda_1, \lambda_2, \ldots, \lambda_s$ be the distinct eigenvalues of $G$ with multiplicity $m_1, m_2, \ldots, m_s$, respectively. The multi-set ${\rm Spec}(G)=\{(\lambda_1)^{m_1}, (\lambda_2)^{m_2}, \ldots, (\lambda_s)^{m_s}\}$ of eigenvalues of $A(G)$ is called the adjacency spectrum of $G$.

For two graphs $G$ and $H$, if ${\rm Spec}(G)={\rm Spec}(H)$, we say $G$ and $H$  are cospectral with respect to adjacency matrix. A graph $G$ is said to be determined by its spectrum or  DS for short, if for a graph $H$ with ${\rm Spec}(G)={\rm Spec}(H)$, one has $G$ is isomorphic to $H$. So far numerous examples of cospectral but non-isomorphic graphs are constructed by  interesting techniques such as Seidel switching, Godsil-McKay switching, Sunada or Schwenk method. For more information, one may see \cite{a1,a2,a12} and the references cited in them. But, only a few graphs with very special structures have been reported to be determined by their spectra (see \cite{a1,a12,a3,a4,a6,b1,b2,b3,b4,b5,b6} and the references cited in them). Also, Wei Wang  and Cheng-Xian Xu have developed a new method in \cite{b7} to show that  many graphs are determined by their spectrum and the spectrum of their complement.

The friendship (or Dutch windmill  or $n$-fan) graph $F_n$ is a graph that can be constructed by coalescence $n$ copies of the cycle graph $C_3$ of length $3$ with a common vertex. By construction, the friendship graph $F_n$ is isomorphic to the windmill graph $Wd\left(3,n\right)$ \cite{a7}. The \textbf{Friendship Theorem} of Paul Erd\"os, Alfred R\' enyi and Vera T. S\' os \cite{a8}, states that  graphs with the property that every two vertices have exactly one neighbour in common are exactly the friendship graphs.
\noindent
The figure \ref{pic1} shows some examples of friendship graphs.
\begin{figure}[htb]
\centering
\begin{tikzpicture}[scale=0.7]
\filldraw [black]
(-1,-0.5) circle (3.5 pt)
(-1,0.5) circle (3.5 pt)
(0,0) circle (3.5 pt)
(1,0.5) circle (3.5 pt)
(1,-0.5) circle (3.5 pt);
\node [label=below:$F_2$] (F_2) at (0,-1.2) {};
\draw[thick] (-1,-0.5) -- (-1,0.5);
\draw[thick] (-1,0.5) -- (0,0);
\draw[thick] (-1,-0.5) -- (0,0);
\draw[thick] (0,0) -- (1,0.5);
\draw[thick] (0,0) -- (1,-0.5);
\draw[thick] (1,0.5) -- (1,-0.5);
\end{tikzpicture} \qquad
\begin{tikzpicture}[scale=0.7]
\filldraw [black]
(-1,-0.5) circle (3.5 pt)
(-1,0.5) circle (3.5 pt)
(0,0) circle (3.5 pt)
(1,0.5) circle (3.5 pt)
(1,-0.5) circle (3.5 pt)
(-0.5,1) circle (3.5 pt)
(0.5,1) circle (3.5 pt);
\node [label=below:$F_3$] (F_3) at (0,-1.2) {};
\draw[thick] (-1,-0.5) -- (-1,0.5);
\draw[thick] (-1,0.5) -- (0,0);
\draw[thick] (-1,-0.5) -- (0,0);
\draw[thick] (0,0) -- (1,0.5);
\draw[thick] (0,0) -- (1,-0.5);
\draw[thick] (1,0.5) -- (1,-0.5);
\draw[thick] (-0.5,1) -- (0.5,1);
\draw[thick] (-0.5,1) -- (0,0);
\draw[thick] (0.5,1) -- (0,0);
\end{tikzpicture}\qquad
\begin{tikzpicture}[scale=0.7]
\filldraw [black]
(-1,-0.5) circle (3.5 pt)
(-1,0.5) circle (3.5 pt)
(0,0) circle (3.5 pt)
(1,0.5) circle (3.5 pt)
(1,-0.5) circle (3.5 pt)
(-0.5,1) circle (3.5 pt)
(0.5,1) circle (3.5 pt)
(-0.5,-1) circle (3.5 pt)
(0.5,-1) circle (3.5 pt);
\node [label=below:$F_4$] (F_4) at (0,-1.2) {};
\draw[thick] (-1,-0.5) -- (-1,0.5);
\draw[thick] (-1,0.5) -- (0,0);
\draw[thick] (-1,-0.5) -- (0,0);
\draw[thick] (0,0) -- (1,0.5);
\draw[thick] (0,0) -- (1,-0.5);
\draw[thick] (1,0.5) -- (1,-0.5);
\draw[thick] (-0.5,1) -- (0.5,1);
\draw[thick] (-0.5,1) -- (0,0);
\draw[thick] (0.5,1) -- (0,0);
\draw[thick] (-0.5,-1) -- (0.5,-1);
\draw[thick] (-0.5,-1) -- (0,0);
\draw[thick] (0.5,-1) -- (0,0);
\end{tikzpicture}\qquad
\begin{tikzpicture}[scale=0.7]
\filldraw [black]
(-1,-0.5) circle (3.5 pt)
(-1,0.5) circle (3.5 pt)
(0,0) circle (3.5 pt)
(-0.5,1) circle (3.5 pt)
(0.5,1) circle (3.5 pt)
(0.8,0.5) circle (1 pt)
(0.9,0) circle (1 pt)
(0.8,-0.5) circle (1 pt)
(-0.5,-1) circle (3.5 pt)
(0.5,-1) circle (3.5 pt);
\node [label=below:$F_n$] (F_n) at (0,-1.2) {};
\draw[thick] (-1,-0.5) -- (-1,0.5);
\draw[thick] (-1,0.5) -- (0,0);
\draw[thick] (-1,-0.5) -- (0,0);
\draw[thick] (-0.5,1) -- (0.5,1);
\draw[thick] (-0.5,1) -- (0,0);
\draw[thick] (0.5,1) -- (0,0);
\draw[thick] (-0.5,-1) -- (0.5,-1);
\draw[thick] (-0.5,-1) -- (0,0);
\draw[thick] (0.5,-1) -- (0,0);
\end{tikzpicture}
\caption{Friendship graphs $F_2$, $F_3$, $F_4$ and $F_n$}\label{pic1}
\end{figure}
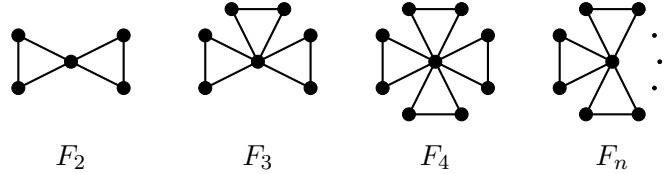

 In \cite{a9} it is  proved that the friendship graphs can be determined by the signless Laplacian spectrum and in \cite{a9,a10} the following conjecture has been proposed:

\textbf{Conjecture 1.} \textit{The friendship graph is $DS$ respect to the adjacency matrix.}

Conjecture 1 has been recently studied in \cite{c1}. It is claimed as \cite[Theorem 3.2]{c1} that the Conjecture 1 is valid. We believe that there is a gap in the proof of \cite[Theorem 3.2]{c1} where the interlacing theorem has been applied for subgraphs of the graph which are not clear if they are induced or not. Therefore, we give our results independently.

The rest of this paper is organized as follows. In Section 2, we obtain some  preliminary results about the cospectral mate of a friendship graph. In this section we prove that if the cospectral mate of $F_n$ is connected and planar then it is isomorphic to $F_n$. In Section 3, it is proved that, if the cospectral mate of $F_n$ is connected and does not have $C_5$ as a subgraph, then it is isomorphic to $F_n$. Also, we prove that, if there are two adjacent vertices with degree 2 in a cospectral mate of $F_n$, then $G$ is isomorphic to $F_n$ and some variations of the latter result is studied. In Section 4, the complement of the cospectral mate is studied and we show that if this complement is disconnected, then the cospectral mate is isomorphic to $F_n$. Also, it is shown that the complement of the friendship graph $F_n$ is DS.


\section{\bf Some Properties of Cospectral Mate of $F_n$}
\vskip 0.4 true cm

We first give  some preliminary facts and theorems which are useful in the sequel. For the proof of these facts one may see \cite{a11}.

\begin{lem}
Let $G$ be a graph. For the adjacency matrix of $G$, the following information can be deduced from the spectrum:
\begin{enumerate}
\item[1.]
The number of vertices
\item[2.]
The number of edges
\item[3.]
The number of closed walks of any length
\item[4.]
Being regular or not and the degree of regularity
\item[5.]
Being bipartite or not.
\end{enumerate}
\end{lem}

\begin{prop}
Let $F_n$ denote the friendship graph with $2n+1$ vertices, then we have
$$Spec(F_n)=\left\lbrace\left(\frac{1}{2}-\frac{1}{2}\sqrt{1+8n}\right)^1,{\left(-1\right)}^n,{\left(1\right)}^{n-1},{\left(\frac{1}{2}+\frac{1}{2}\sqrt{1+8n}\right)}^1 \right\rbrace.$$
\end{prop}
\begin{proof}
The friendship graph $F_n$ with $2n+1$ vertices is the cone of the disjoint union of $n$ complete graphs $K_2$: $ K_1\nabla nK_2$. It follows from Theorem 2.1.8 of \cite{a11} that the characteristic polynomial of $F_n$ is:
$$P_{F_n}(x)=\left(x+1\right){\left(x^2-1\right)}^{n-1}\left(x^2-x-2n\right).$$
Therefore the spectrum of $F_n$ is as claimed in the theorem.
\end{proof}

Let $H$ be any graph. A graph $G$ is called $H$-free if it does not have  an induced subgraph isomorphic to $H$. In the following, we examine the structure of $G$ as a cospectral graph of $F_n$.

\begin{prop}\label{Prop2,3}
Let the graph $G$ be cospectral with friendship graph $F_n$. Then we have:
\begin{enumerate}
\item[1.]
If $H$ is a graph with ${\lambda }_2\left(H\right)>1$, then $G$ is $H$-free.
\item[2.]
If $H$ is a graph having at least two negative eigenvalues less than $-1$, then $G$ is $H$-free.
\end{enumerate}
\end{prop}
\begin{proof}
We know that $\ {\lambda }_2\left(F_n\right)=1$ and $F_n$ has only one eigenvalue less than $-1$. Therefore by using interlacing theorem for induced subgraph $H$ that satisfies the conditions of above proposition, $G$ must be $H$-free.
\end{proof}

\begin{thm}\label{Theorem1}
Let graph $G$ be cospectral with friendship graph $F_n$. Then $G$ is either connected or it is the disjoint union of some $K_2$ and a connected component.
\end{thm}
\begin{proof}
Suppose $G$ is cospectral with friendship graph $F_n$ and $G=G_1\cup G_2\cup \ldots \cup G_m$. It is easy to see that $\lambda_2(K_3\cup P_3)$ and $\lambda_2(K_3\cup K_3)$ is greater than $1$. Therefore, none of the components of $G$, excluding say $G_1$, does not contain any $K_3$ and $P_3$ as an induced subgraph. So, if $G$ is not connected, the components of $G$, except $G_1$, must be isomorphic to $K_2$, since $G$ does not have any isolated vertices.
\end{proof}

\begin{defn}\cite{a14}
A graph is triangulated if it has no chordless induced cycle with four or more vertices. It follows that the complement of a
 triangulated graph can not contain a chordless cycle with five or more vertices.
\end{defn}

\begin{prop}\label{tria}
Let graph $G$ be connected, planar and cospectral with friendship graph $F_n$, then $G$ is triangulated.
\end{prop}
\begin{proof}
The graph $G$ is planar and connected with $n$ triangles, $2n+1$ vertices and $3n$ edges.
Also, the number of faces is an invariant parameter between two cospectral connected planar graphs,
since it only depends on the number of vertices and edges. Let $\textit{f}(G)$ denotes the number of faces of graph $G$.
By Euler formula for connected planar graphs, $\textit{f}(G)=2-|V(G)|+|E(G)|$, the number of faces of $G$ is $n+1$ and each
inner face of $G$ must be an induced triangle. Therefore $G$ has no chordless induced cycle with four or more vertices.
\end{proof}

In the following, we express an interesting theorem proved by Vladimir Nikiforov in \cite{V.niki}, and we use it to prove some results.

\begin{thm}\label{Niki}
\cite{V.niki}
Let $G$ be a graph of order $n$ with $\lambda_1(G)=\lambda$. If $G$ has no 4-cycles, then
$$\lambda^2-\lambda \geq n-1,$$
and equality holds if and only if every two vertices of $G$ have exactly one common neighbour.
\end{thm}

\begin{cor}\label{trian}
Let $G$ be connected, planar and cospectral with friendship graph $F_n$. Then $G$ is isomorphic to $F_n$.
\end{cor}
\begin{proof}
By Proposition \ref{tria}, the graph $G$ is $C_4$-free and $\lambda_1^2(G)-\lambda_1(G)=2n$. Therefore by Theorem\ref{Niki}, the graph $G$ must be isomorphic to $F_n$.
\end{proof}

Suppose $\chi(G)$ and $\omega(G)$ are chromatic number and clique number of graph $G$, respectively. A graph $G$ is \textbf{\textit{perfect}} if $\chi(H)=\omega(H)$ for every induced subgraph $H$ of $G$. It is proved that a graph $G$ is perfect if and only if $G$ is \textit{Berge}, that is, it contains no odd hole or antihole, where odd hole and antihole are odd cycle, $C_m$ for $m\geq5$, and its complement, respectively.  Also in 1972 \textit{Lov\' asaz} proved that, A graph is perfect if and only if its complement is perfect \cite{a15}.

\begin{prop}
Let graph $G$ be cospectral with friendship graph $F_n$. Then $G$ and $\overline{G}$ are perfect.
\end{prop}
\begin{proof}
The spectrum of hole, or undirected n-cycle $C_n$ for $n$ odd and $n\geq 5$, is $2cos(2\pi j/n)$ for $(j=0,1,\ldots,n-1)$. It is easy to check that for $n$ odd and $n\geq 5$, $\lambda_{n-1}(C_n)$ and $\lambda_{n}(C_n)$ is strictly less than $-1$. Therefore by Proposition \ref{Prop2,3}, it can not be an induced subgraph of $G$, so $G$ does not have any hole. Also, since the spectrum of antihole, or complement of hole, are $n-3$ and $-1-\lambda_j$ for $(j=1,\ldots,n-1)$, then we have at least two negative eigenvalues less than $-1$, and it shows that $G$ does not have any antihole. So, $G$  and also $\overline{G}$ are perfect graphs.
\end{proof}

\begin{prop}
Let $G$ be connected and cospectral with friendship graph $F_n$. Then each connected induced subgraph of $G$ contains a
dominating (not necessarily induced) complete bipartite graph. Moreover, we can find such a dominating subgraph in polynomial time.
\end{prop}
\begin{proof}
By \cite[Theorem 6]{a16}, and Proposition \ref{Prop2,3}, we only must prove that $G$ does not have induced dominating $C_6$ as a subgraph.
Suppose $G$ has an induced dominating $C_6$ subgraph. By hand checking and interlacing theorem, it is easy to see that the seventh vertex of $G$ must join to three non adjacent vertices of $C_6$. Also, for each $2n+1-7=2n-6$ remaining vertices, we have at least two edges. So, the total number of edges in $G$ is at least $2(2n-6)+3+6=4n-3$, that is contradiction.
\end{proof}

\section{\bf  Structural Properties of Cospectral Mates of $F_n$}
\vskip 0.4 true cm

It can be seen by Theorem \ref{Niki} that, if $G$ is cospectral with $F_n$ and does not have $C_4$ or induced $C_4$, then $G$ is isomorphic to $F_n$. In the following, we study the cospectral mate of $F_n$ with respect to the $C_5$ subgraph. Also, We know that in the graph $F_n$, and also in its cospectral mate, the average number of triangles containing a given vertex is $\frac{6n}{4n+2}$, that is strictly greater than one. By this evidence, we obtain some interesting results about the cospectral mate of $F_n$.

Firstly, we will prove that; if $G$ is connected, cospectral with $F_n$ and does not have $C_5$ as a subgraph, then $G$ is isomorphic to $F_n$.

\begin{lem}\label{1}
\cite{a11}
Let the graph $G$ is connected and $H$ is a proper subgraph of $G$. Then
$$\lambda_{max}(H)<\lambda_{max}(G).$$
\end{lem}

\begin{lem}\label{2}
Let $G$ be a connected graph that is cospectral with $F_n$ and $\delta(G)$ be the minimum degree of $G$. Then, $\delta(G)=2$ and $G$ has at least three vertices with this minimum degree.
\end{lem}
\begin{proof}
By contradiction, suppose $G$ has at least one vertex with degree $1$, say $v$, that is adjacent with vertex $w$ of $G$. The graphs $G\setminus\{v\}$ and $G\setminus\{v,w\}$ are induced subgraphs of $G$. Let $\mu_1\geq \mu_2\geq\ldots \geq \mu_{2n-1}$ be the eigenvalues of graph $G\setminus\{v,w\}$, then by Interlacing theorem we have:
$$\lambda_j \geq \mu_j \geq \lambda_{j+2}, (j=1, 2, \ldots, 2n-1)$$
where $\lambda_i$, $(i=1, 2, \ldots, 2n+1)$ are the eigenvalues of $G$. So, we deduce that:
\begin{enumerate}
\item[\textit{i})]
$|\mu_1|\leq |\lambda_1|=\left|\frac{1+\sqrt{8n+1}}{2}\right|$,
\item[\textit{ii})]
$|\mu_{2n-1}|\leq |\lambda_{2n+1}|=\left|\frac{1-\sqrt{1+8n}}{2}\right|$,
\item[\textit{iii})]
$|\mu_j|\leq 1 (j=2, 3, \ldots, 2n-2)$.
\end{enumerate}
Now Theorem 2.2.1 of \cite{a11} implies that $$P_G(x)=xP_{G\setminus\{v\}}(x)-P_{G\setminus\{v,w\}}(x),$$ It follows that $$(-1)^{n}(2n)=P_G(0)=P_{G\setminus\{v,w\}}(0)=\prod_{j=1}^{2n-1}{\mu_j}.$$
Therefore, by $(i), (ii)$ and $(iii)$ we obtain $\mu_1=\lambda_1$, that is a contradiction with Lemma \ref{1}. So, graph $G$ does not have any vertex of degree one.

Now, suppose $G$ has $t$ vertices of degree $2$. Therefore we have $$2t+3(2n+1-t)\leq\sum_{i=1}^{2n+1}{deg(v_i(G))=6n},$$
It implies that $t\geq3$, $\delta(G)=2$ and $G$ has at least three vertices of degree two.
\end{proof}

\begin{rem}
By same arguments, we can show that if $G$ is not connected, then the component of $G$ that is not isomorphic to $K_2$ does not have any vertex with degree one. But in this case, the minimum degree of $G$ can not be determine, since it depends to the number of components that is isomorphic to $K_2$.
\end{rem}

\begin{lem}\label{4}
Let $G$ be graph of order $n$ and degree sequence $d_1, d_2, \ldots, d_n$. Then, the number of closed walks of length five in $G$ is given by$$tr(A^5(G))=\sum_{i=1}^{n}{\lambda_i^5(G)}=30N_G(C_3)+10N_G(C_5)+10N_G(C_3^*),$$ where $C_3^*$ is isomorphic to $K_3$ with one pendant.
\end{lem}
\begin{proof}
It is easy to see that, the only subgraphs of $G$ that participate in counting of closed walk of length five are, $C_3$, $C_5$ and $C_3^*$. By summation the fifth power of the eigenvalues of these graphs, the coefficients of $N_G(C_3), N_G(C_5)$ and $N_G(C_3^*)$ must be $30, 10$ and $40$, respectively. But, $C_3$ is counted twice in $G$ and $C_3^*$, $30$ times for each $C_3^*$, and we have to subtract it. Therefore, we obtain the equation that is mentioned in lemma.
\end{proof}

\begin{lem}\label{5}
Suppose the function $S$, $S:\mathbb{R}^n\longmapsto\mathbb{R}$, be defined as $S(x_1,x_2,\ldots,x_n)=\sum_{i=1}^{n}{x_i^2}$, and $\sum_{i=1}^{n}{x_i}=M$. Then we have
\begin{enumerate}
\item[\textit{i})]
If $x_i\geq0$  $(i=1,2,\ldots,n)$, then the maximum of $S$ is $M^2$ and this value only happens in $M\textbf{e}_i$, where $\{\textbf{e}_1, \textbf{e}_2, \ldots, \textbf{e}_n\}$ is the standard orthonormal basis of $\mathbb{R}^n$.
\item[\textit{ii})]
If $x_i\geq d$  $(i=1,2,\ldots,n)$, then the maximum of $S$ is $(n-1)d^2+(M-(n-1)d)^2$ and this value only happens in $(M-(n-1)d)\textbf{e}_i+d\textbf{j}$, where $\textbf{j}$ denotes the all-1 vector of size $1\times n$.
\end{enumerate}
\end{lem}
\begin{proof}
Let $T=2\sum_{1\leq i<j\leq n}{x_ix_j}$. Now for proving case $(i)$, it suffice to note that $S=M^2-T$, and for maximizing the function $S$, we must minimize the function $T$. But the minimum of $T$ is zero and it happens only in $M\textbf{\textit{e}}_i$, since we have  $\sum_{i=1}^{n}{x_i}=M$.\\
For proving case $(ii)$, let $y_i=x_i-d$ $(i=1, 2, \ldots, n)$ and $T(y_1,y_2,\ldots,y_n)=\sum_{i=1}^{n}{y_i^2}$. So, $y_i\geq 0$ and $\sum_{i=1}^{n}{y_i}=M-nd$. Now by using part $(i)$, the maximum of $T$ is $(M-nd)^2$ and it only happens in $(M-nd)\textbf{\textit{e}}_i$. Therefore, by backing the changed variables and the fact $S=T-nd^2+2Md$, we obtain the mentioned results.
\end{proof}

\begin{lem}\label{6}
Suppose the function $S$, $S:\mathbb{R}^n\longmapsto\mathbb{R}$, is defined as $S(x_1,x_2,\ldots,x_n)=\sum_{i=1}^{n}{t_ix_i}$, such that $t_i$ $(i=1, 2, \ldots, n)$ are real numbers, $\sum_{i=1}^{n}{x_i}=M$ and $x_i\geq d$. If we have $0\leq t_1\leq t_2\leq \ldots \leq t_n$, then the maximum of function $S$ is $d(t_1+t_2+\ldots+t_{n-1})+t_n(M-(n-1)d)$.
\end{lem}
\begin{proof}
Since the real numbers $t_i$ $(i=1,2,\ldots,n)$ are increasing, by Lemma \ref{5} the result is clear.
\end{proof}

\begin{thm}
Suppose the graph $G$ is connected and cospectral with friendship graph $F_n$. If $G$ does not have $C_5$ as a subgraph, then $G$ is isomorphic to $F_n$.
\end{thm}
\begin{proof}
The graph $G$ is cospectral with $F_n$, so the number of vertices, edges, triangles and closed walk of length 5 are equal in these both graphs. By Lemma \ref{4}, we have $N_G(C_3^*)=N_{F_n}(C_3^*)$. Now, we calculate the number of $N_G(C_3^*)$ in two ways. Suppose $v_1, v_2, \ldots, v_{2n+1}$, are the vertices of graph $G$. Let $t_i$  $(i=1,2,\ldots,2n+1)$, shows the number of triangles that has vertex $v_i$. So, the total number of $C_3^*$ that has $v_i$ as a vertex with degree three is $t_i(deg_G(v_i)-2)$. Therefore, we have
\begin{align}
N_G(C_3^*)=\sum_{i=1}^{2n+1}{t_i(deg_G(v_i)-2)}=\sum_{i=1}^{2n+1}{t_ideg_G(v_i)}-2\sum_{i=1}^{2n+1}{t_i}.\label{one}
\end{align}
 But we have
\begin{align}
\sum_{i=1}^{2n+1}{t_i}=3N_G(C_3)=3n\Longrightarrow N_G(C_3^*)=\sum_{i=1}^{2n+1}{t_ideg_G(v_i)-6n} \label{two}
\end{align}
Since $N_{F_n}(C_3^*)=2n^2+4n-6n$, by (\ref{one}) and (\ref{two}) we obtain
\begin{align}
\sum_{i=1}^{2n+1}{t_ideg_G(v_i)}=2n^2+4n \label{three}
\end{align}
Now we claim that $G$ is isomorphic to $F_n$. Suppose $x_i\geq 2, y_i\geq 0$  $(i=1,2,\ldots,2n+1)$, $\sum_{i=1}^{2n+1}{x_i}=3n$, $\sum_{i=1}^{2n+1}{y_i}=6n$ and define the function $F$ as follow $$F(x_1,x_2,\ldots ,x_{2n+1},y_1,y_2,\ldots ,y_{2n+1})=\sum_{i=1}^{2n+1}{x_iy_i}.$$
We show that, if $(x_1,\ldots,x_{2n+1})=(t_1,\ldots,t_{2n+1})$ and $(y_1,\ldots,y_{2n+1})=(deg_G(v_1),\ldots,deg_G(v_{2n+1}))$, then the maximum of function $F$ is happen for the graph $F_n$.

Let $\mathcal{A}=\{G_1,G_2,\ldots,G_k\}$ be the set of all connected graphs with $2n+1$ vertices, $3n$ edges, $n$ triangles, minimum degree 2 and without any $C_5$ subgraph. The vertices of $G_i$ $(i=1,2,\ldots,k)$ can be labelled in such a way that, for each graph $G_i$ we have $t_1\leq t_2\leq \ldots \leq t_{2n+1}$. It is easy to see that, $F_n$ is an element of $\mathcal{A}$. Now, we want to find the maximum of $\sum_{i=1}^{2n+1}{t_ideg_G(v_i)}$ among the members of $\mathcal{A}$. We prove that, the maximum value of $\sum_{i=1}^{2n+1}{t_ideg_G(v_i)}$ is equal to $2n^2+4n$ and it only happens for the graph $F_n$.

For each graph $G\in \mathcal{A}$, let $X_G=(t_1,t_2,\ldots,t_{2n+1})$, $Y_G=(deg_G(v_1),deg_G(v_2),\ldots,deg_G(v_{2n+1}))$ and $F(G)=F(X_G,Y_G)=X_G\cdot Y_G$. It is clear that $F(F_n)=2n^2+4n$. By Lemma \ref{6}, for each graph $G\in \mathcal{A}$ we have $$F(G)=t_1deg_G(v_1)+\ldots +t_{2n+1}deg_G(v_{2n+1})\leq 2t_1+\ldots +2nt_{2n+1}.$$
But the latter inequality implies that, for each graph $G\in \mathcal{A}$ we have $$F(G)\leq F(X_G,Y_0),$$ where $Y_0=(2,2,\ldots,2n)$. Among the members of $\mathcal{A}$, the only graph that has $Y_0$ as a degree sequence is $F_n$. Therefore, the graph $G$ is isomorphic to $F_n$, since by (\ref{three}) we have $\sum_{i=1}^{2n+1}{t_ideg_G(v_i)}=F(F_n)$.
\end{proof}

In the next proposition, we show that the cospectral mate of friendship graph has a lot of vertices of degree two.

\begin{prop}\label{29}
Let the graph $G$ is cospectral with friendship graph $F_n$, and let $d_2(G)$ and $\triangle(G)$ be the number of vertices of degree two and maximum degree of $G$, respectively. Then we have
\begin{enumerate}
\item[\textit{i})]
If $G$ be disconnected, $G=mK_2\cup G_1$, then $$m\leq \frac{\lambda_{max}|V(G)|-2|E(G)|}{-2\lambda_{min}}.$$ Moreover, if $d_2(G_1)\neq 0$, then $d_2(G_1)\geq \lambda_{max}-4m$.
\item[\textit{ii})]
If $G$ be connected, then $d_2(G)\geq 1+\lambda_{max}$.
\end{enumerate}
\end{prop}
\begin{proof}
It is easy to see that, if $G=mK_2\cup G_1$ then the component $G_1$ has $3n-m$ edges, $2n+1-2m$ vertices, $\lambda_{max}=\lambda_1(F_n)$ and $\lambda_{min}=\lambda_{2n+1}(F_n)$. Now by Theorem 3.2.1 of \cite{a11}, we have $$\frac{2(3n-m)}{2n+1-2m}\leq \lambda_{max},$$ by simplification and using $\lambda_{max}-1=-\lambda_{min}$, we proved the first part of $(i)$. Again, by using Theorem 3.2.1 of \cite{a11} and $\sum_{i=1}^{2n+1}{deg_{G}(v_i)=6n}$, we obtain $$2m+2d_2(G_1)+\lambda_{max}+3(2n+1-2m-t-1)\leq 6n,$$ so by simplification, the second part of $(i)$ is proved.

For proving part $(ii)$, we notice that the graph $G$ is not regular. So by Theorem 3.2.1 of \cite{a11} we obtain $\triangle(G)\geq 1+\lambda_{max}$. Now, we have to have $$2d_2(G)+1+\lambda_{max}+3(2n+1-d_2(G)-1)\leq 6n.$$ By simplification, we obtain the requested result.
\end{proof}

In the following, we obtain some structural properties of the cospectral mate of friendship graph. Actually, these results are some good evidences to show that the friendship graph is $DS$.

\begin{defn}
Suppose $G$ is a graph and $H$ is its subgraph. If $x$ be a vertex of $H$ with degree $r$ in $G$, we denote it by $d_G(x)=r$.
\end{defn}

\begin{lem}\label{Tx}
Let the graph $G$ is cospectral with $F_n$ and $G$ has a subgraph $K_3$ which has two vertices of degree $2$ in $G$. Then $G$ is isomorphic to $F_n$.
\end{lem}
\begin{proof}
Suppose the $K_3$ that is mentioned in the above Lemma has vertices $\{x,y,z\}$, where $d_G(x)=d_G(y)=2$, and denote it by $T_{x,y}$. We prove that, an arbitrary triangle of $G$ must share a common vertex with $T_{x,y}$ at vertex $z$. Let $\{u,v,w\}$ be the vertices of an arbitrary triangle in $G$. At least one vertex of this triangle is joined to the vertex $z$, since $G$ is $2K_3$-free. Therefore, the all cases that can happen are shown in Figure \ref{pic2}. But the graph $G$ is $\{A_2,A_3,A_4\}$-free, since $\lambda_2(A_2)=1.73205$, $\lambda_2(A_3)=1.50694$ and $\lambda_2(A_4)=1.33988$. Now, we prove that, there is not any edge between other $n-1$ triangles in $G$. Suppose there are two triangles in $G$ with some edges between them. Since these two triangles have a common vertex in $z$ and all triangles in $G$ also must have, the all possible cases showed in Figure \ref{pic3}. But $G$ is $\{B_1,B_2\}$-free, since $\lambda_2(B_1)=1.19799$ and $\lambda_2(B_2)=1.28917$. This is contradiction and so, there are not any edges between other $n-1$ triangles in $G$. So, $G$ is isomorphic to $F_n$.
\begin{figure}[htb]
\centering
\begin{tikzpicture}[scale=0.7]
\filldraw [black]
(-1,-0.5) circle (3.5 pt)
(-1,0.5) circle (3.5 pt)
(0,0) circle (3.5 pt)
(1,0.5) circle (3.5 pt)
(1,-0.5) circle (3.5 pt);
\node [label=below:$A_1$] (A_1) at (0,-0.75) {};
\node [label=left:$y$] (y) at (-1,-0.5) {};
\node [label=left:$x$] (x) at (-1,0.5) {};
\node [label=left:$z$] (z) at (0,0) {};
\node [label=right:$u$] (u) at (0,0) {};
\node [label=right:$v$] (v) at (1,0.5) {};
\node [label=right:$w$] (w) at (1,-0.5) {};
\draw[thick] (-1,-0.5) -- (-1,0.5);
\draw[thick] (-1,0.5) -- (0,0);
\draw[thick] (-1,-0.5) -- (0,0);
\draw[thick] (0,0) -- (1,0.5);
\draw[thick] (0,0) -- (1,-0.5);
\draw[thick] (1,0.5) -- (1,-0.5);
\end{tikzpicture}
\begin{tikzpicture}[scale=0.7]
\filldraw [black]
(-1.5,-0.5) circle (3.5 pt)
(-1.5,0.5) circle (3.5 pt)
(-0.5,0) circle (3.5 pt)
(0.5,0) circle (3.5 pt)
(1.5,0.5) circle (3.5 pt)
(1.5,-0.5) circle (3.5 pt);
\node [label=below:$A_2$] (A_2) at (0,-0.75) {};
\node [label=left:$y$] (y) at (-1.5,-0.5) {};
\node [label=left:$x$] (x) at (-1.5,0.5) {};
\node [label=left:$z$] (z) at (-0.5,0) {};
\node [label=right:$u$] (u) at (0.5,0) {};
\node [label=right:$v$] (v) at (1.5,0.5) {};
\node [label=right:$w$] (w) at (1.5,-0.5) {};
\draw[thick] (-1.5,-0.5) -- (-1.5,0.5);
\draw[thick] (-1.5,-0.5) -- (-0.5,0);
\draw[thick] (-1.5,0.5) -- (-0.5,0);
\draw[thick] (-0.5,0) -- (0.5,0);
\draw[thick] (0.5,0) -- (1.5,0.5);
\draw[thick] (0.5,0) -- (1.5,-0.5);
\draw[thick] (1.5,0.5) -- (1.5,-0.5);
\end{tikzpicture}
\begin{tikzpicture}[scale=0.7]
\filldraw [black]
(-1.5,-0.5) circle (3.5 pt)
(-1.5,0.5) circle (3.5 pt)
(-0.5,0) circle (3.5 pt)
(0.5,0) circle (3.5 pt)
(1.5,0.5) circle (3.5 pt)
(1.5,-0.5) circle (3.5 pt);
\node [label=below:$A_3$] (A_3) at (0,-0.75) {};
\node [label=left:$y$] (y) at (-1.5,-0.5) {};
\node [label=left:$x$] (x) at (-1.5,0.5) {};
\node [label=left:$z$] (z) at (-0.5,0) {};
\node [label=right:$u$] (u) at (0.5,0) {};
\node [label=right:$v$] (v) at (1.5,0.5) {};
\node [label=right:$w$] (w) at (1.5,-0.5) {};
\draw[thick] (-1.5,-0.5) -- (-1.5,0.5);
\draw[thick] (-1.5,-0.5) -- (-0.5,0);
\draw[thick] (-1.5,0.5) -- (-0.5,0);
\draw[thick] (-0.5,0) -- (0.5,0);
\draw[thick] (0.5,0) -- (1.5,0.5);
\draw[thick] (0.5,0) -- (1.5,-0.5);
\draw[thick] (1.5,0.5) -- (1.5,-0.5);
\draw[thick] (1.5,0.5) -- (-0.5,0);
\end{tikzpicture}
\begin{tikzpicture}[scale=0.7]
\filldraw [black]
(-1.5,-0.5) circle (3.5 pt)
(-1.5,0.5) circle (3.5 pt)
(-0.5,0) circle (3.5 pt)
(0.5,0) circle (3.5 pt)
(1.5,0.5) circle (3.5 pt)
(1.5,-0.5) circle (3.5 pt);
\node [label=below:$A_4$] (A_4) at (0,-0.75) {};
\node [label=left:$y$] (y) at (-1.5,-0.5) {};
\node [label=left:$x$] (x) at (-1.5,0.5) {};
\node [label=left:$z$] (z) at (-0.5,0) {};
\node [label=right:$u$] (u) at (0.5,0) {};
\node [label=right:$v$] (v) at (1.5,0.5) {};
\node [label=right:$w$] (w) at (1.5,-0.5) {};
\draw[thick] (-1.5,-0.5) -- (-1.5,0.5);
\draw[thick] (-1.5,-0.5) -- (-0.5,0);
\draw[thick] (-1.5,0.5) -- (-0.5,0);
\draw[thick] (-0.5,0) -- (0.5,0);
\draw[thick] (0.5,0) -- (1.5,0.5);
\draw[thick] (0.5,0) -- (1.5,-0.5);
\draw[thick] (1.5,0.5) -- (1.5,-0.5);
\draw[thick] (1.5,0.5) -- (-0.5,0);
\draw[thick] (1.5,-0.5) -- (-0.5,0);
\end{tikzpicture}
\caption{All possible cases between $T_{x,y}$ and another triangle in $G$}\label{pic2}
\end{figure}
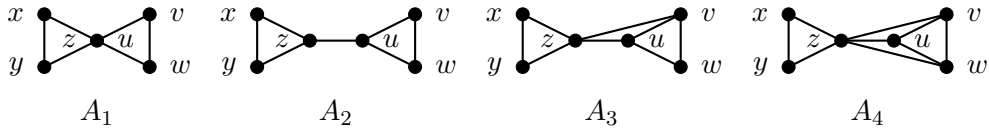
\begin{figure}[htb]
\centering
\begin{tikzpicture}[scale=0.7]
\filldraw [black]
(-1,-0.5) circle (3.5 pt)
(-1,0.5) circle (3.5 pt)
(0,0) circle (3.5 pt)
(1,0.5) circle (3.5 pt)
(1,-0.5) circle (3.5 pt)
(-0.5,1) circle (3.5 pt)
(0.5,1) circle (3.5 pt);
\node [label=below:$B_1$] (B_1) at (0,-0.75) {};
\node [label=left:$x$] (x) at (-1,0.5) {};
\node [label=left:$y$] (y) at (-1,-0.5) {};
\node [label=left:$z$] (z) at (0,0) {};
\draw[thick] (-1,-0.5) -- (-1,0.5);
\draw[thick] (-1,0.5) -- (0,0);
\draw[thick] (-1,-0.5) -- (0,0);
\draw[thick] (0,0) -- (1,0.5);
\draw[thick] (0,0) -- (1,-0.5);
\draw[thick] (1,0.5) -- (1,-0.5);
\draw[thick] (-0.5,1) -- (0.5,1);
\draw[thick] (-0.5,1) -- (0,0);
\draw[thick] (0.5,1) -- (0,0);
\draw[thick] (0.5,1) -- (1,0.5);
\end{tikzpicture}\qquad
\begin{tikzpicture}[scale=0.7]
\filldraw [black]
(-1,-0.5) circle (3.5 pt)
(-1,0.5) circle (3.5 pt)
(0,0) circle (3.5 pt)
(1,0.5) circle (3.5 pt)
(1,-0.5) circle (3.5 pt)
(-0.5,1) circle (3.5 pt)
(0.5,1) circle (3.5 pt);
\node [label=below:$B_2$] (B_2) at (0,-0.75) {};
\node [label=left:$x$] (x) at (-1,0.5) {};
\node [label=left:$y$] (y) at (-1,-0.5) {};
\node [label=left:$z$] (z) at (0,0) {};
\draw[thick] (-1,-0.5) -- (-1,0.5);
\draw[thick] (-1,0.5) -- (0,0);
\draw[thick] (-1,-0.5) -- (0,0);
\draw[thick] (0,0) -- (1,0.5);
\draw[thick] (0,0) -- (1,-0.5);
\draw[thick] (1,0.5) -- (1,-0.5);
\draw[thick] (-0.5,1) -- (0.5,1);
\draw[thick] (-0.5,1) -- (0,0);
\draw[thick] (0.5,1) -- (0,0);
\draw[thick] (-0.5,1) -- (1,-0.5);
\draw[thick] (-0.5,1) -- (1,-0.5);
\draw[thick] (0.5,1) -- (1,0.5);
\end{tikzpicture}
\caption{All possible cases between $T_{x,y}$ and two other triangles in $G$}\label{pic3}
\end{figure}
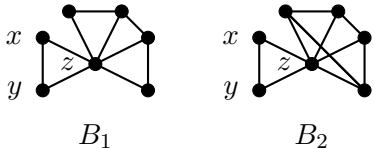
\end{proof}

\begin{thm}
Let the graph $G$ is cospectral with $F_n$ and $G$ has two adjacent vertices of degree $2$. Then $G$ is isomorphic to $F_n$.
\end{thm}
\begin{proof}
Suppose $\{x,y\}$ are two adjacent vertices of degree $2$ in $G$. If these two vertices be adjacent to a vertex $z$ in $G$, then we have a triangle in $G$ with vertices $\{x,y,z\}$. So, by Lemma \ref{Tx}, the result is clear. We show that, this is the only possible case. Now by contrary, suppose the vertices $x$ and $y$ are adjacent to vertices $a$ and $b$, respectively. So, we have a $P_4$ with vertices $\{a,x,y,b\}$ as a subgraph of $G$. Therefore, at least one of the two cases in Figure \ref{pic4} must be happen. First, we examine the graph $C$ of Figure \ref{pic4}. For an arbitrary $K_3$ (or triangle) in $G$, we have $\lambda_2(C\cup K_3)=2$ and so, the all possible cases that $C$ and $K_3$ can construct, are shown in Figure \ref{pic5}. Except the graphs $C_1$ and $C_5$ that have two eigenvalues less than $-1$, for all other graphs, $C_i (i=2,\ldots,26)(i\neq 5)$, we have $\lambda_2(C_i)>1$. Therefore the case $C$ can not happen in $G$.

Now, we examine the graph $D$ of Figure \ref{pic4}. In this case, the graph $D$ is an induced $P_4$ in $G$. For an arbitrary $K_3$ (or triangle) in $G$, we have $\lambda_2(D\cup K_3)=1.61803$ and so, the all possible cases that $D$ and $K_3$ can construct, are shown in Figure \ref{pic6}. Except the graphs $D_3$ that has two eigenvalues less than $-1$, for all other graphs, $D_i (i=1,\ldots,20)(i\neq 3)$, we have $\lambda_2(D_i)>1$. Therefore, the case $D$ can not happen in $G$.\\
These results shows that, if there are two adjacent vertices of degree $2$ in the graph $G$, then they are adjacent to a common vertex in $G$, and this fact complete the proof.
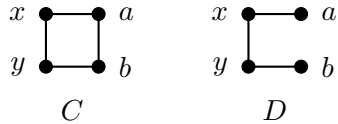
\begin{figure}[htb]
\centering
\begin{tikzpicture}[scale=0.7]
\filldraw [black]
(-0.5,-0.5) circle (3.5 pt)
(-0.5,0.5) circle (3.5 pt)
(0.5,0.5) circle (3.5 pt)
(0.5,-0.5) circle (3.5 pt);
\node [label=below:$C$] (C) at (0,-0.75) {};
\node [label=left:$x$] (x) at (-0.5,0.5) {};
\node [label=left:$y$] (y) at (-0.5,-0.5) {};
\node [label=right:$a$] (a) at (0.5,0.5) {};
\node [label=right:$b$] (b) at (0.5,-0.5) {};
\draw[thick] (-0.5,-0.5) -- (-0.5,0.5);
\draw[thick] (-0.5,0.5) -- (0.5,0.5);
\draw[thick] (-0.5,-0.5) -- (0.5,-0.5);
\draw[thick] (0.5,0.5) -- (0.5,-0.5);
\end{tikzpicture}\qquad
\begin{tikzpicture}[scale=0.7]
\filldraw [black]
(-0.5,-0.5) circle (3.5 pt)
(-0.5,0.5) circle (3.5 pt)
(0.5,0.5) circle (3.5 pt)
(0.5,-0.5) circle (3.5 pt);
\node [label=below:$D$] (D) at (0,-0.75) {};
\node [label=left:$x$] (x) at (-0.5,0.5) {};
\node [label=left:$y$] (y) at (-0.5,-0.5) {};
\node [label=right:$a$] (a) at (0.5,0.5) {};
\node [label=right:$b$] (b) at (0.5,-0.5) {};
\draw[thick] (-0.5,-0.5) -- (-0.5,0.5);
\draw[thick] (-0.5,0.5) -- (0.5,0.5);
\draw[thick] (-0.5,-0.5) -- (0.5,-0.5);
\end{tikzpicture}
\caption{All possible cases for $P_4$ with vertices $\{a,x,y,b\}$ in $G$}\label{pic4}
\end{figure}

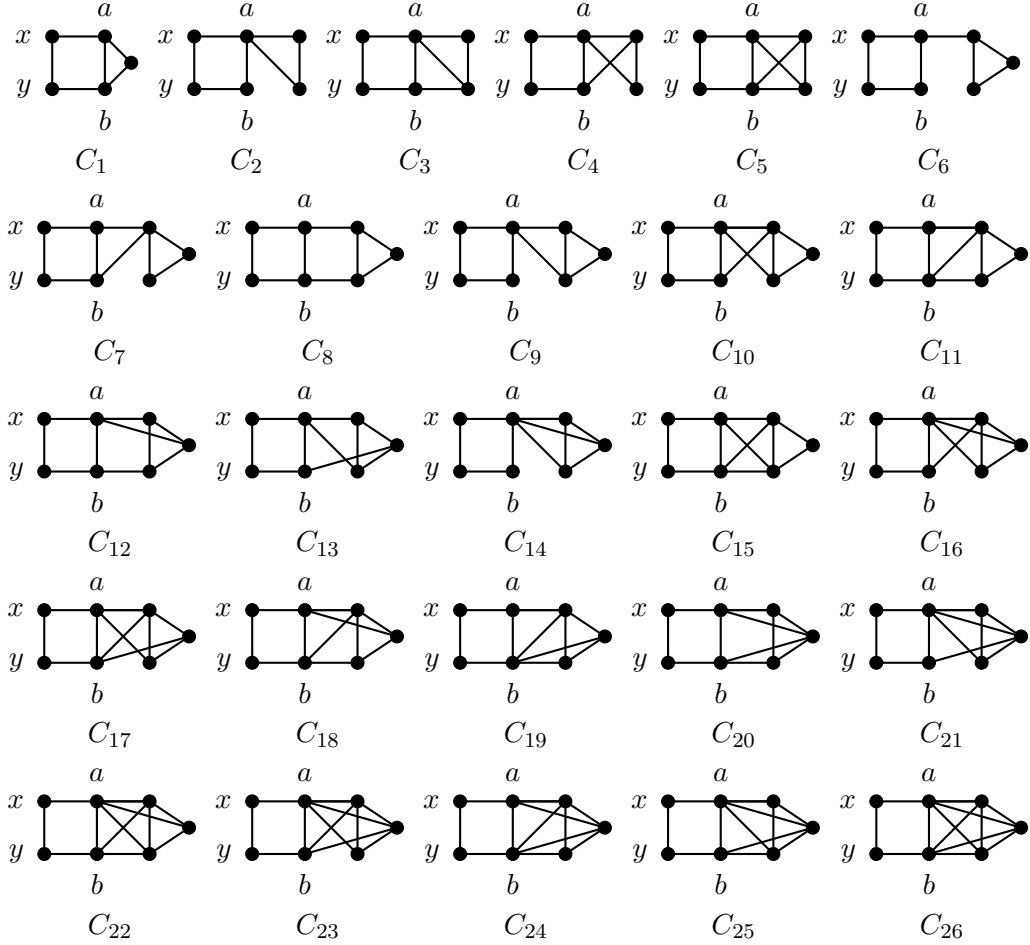
\begin{figure}[htb]
\centering
\begin{tikzpicture}[scale=0.7]
\filldraw [black]
(-0.5,-0.5) circle (3.5 pt)
(-0.5,0.5) circle (3.5 pt)
(0.5,0.5) circle (3.5 pt)
(0.5,-0.5) circle (3.5 pt)
(1,0) circle (3.5 pt);
\node [label=below:$C_1$] (C_1) at (0.25,-1.25) {};
\node [label=left:$x$] (x) at (-0.5,0.5) {};
\node [label=left:$y$] (y) at (-0.5,-0.5) {};
\node [label=above:$a$] (a) at (0.5,0.5) {};
\node [label=below:$b$] (b) at (0.5,-0.5) {};
\draw[thick] (-0.5,-0.5) -- (-0.5,0.5);
\draw[thick] (-0.5,0.5) -- (0.5,0.5);
\draw[thick] (-0.5,-0.5) -- (0.5,-0.5);
\draw[thick] (0.5,0.5) -- (0.5,-0.5);
\draw[thick] (0.5,0.5) -- (1,0);
\draw[thick] (0.5,-0.5) -- (1,0);
\end{tikzpicture}
\begin{tikzpicture}[scale=0.7]
\filldraw [black]
(-0.5,-0.5) circle (3.5 pt)
(-0.5,0.5) circle (3.5 pt)
(0.5,0.5) circle (3.5 pt)
(0.5,-0.5) circle (3.5 pt)
(1.5,-0.5) circle (3.5 pt)
(1.5,0.5) circle (3.5 pt);
\node [label=below:$C_2$] (C_2) at (0.5,-1.25) {};
\node [label=left:$x$] (x) at (-0.5,0.5) {};
\node [label=left:$y$] (y) at (-0.5,-0.5) {};
\node [label=above:$a$] (a) at (0.5,0.5) {};
\node [label=below:$b$] (b) at (0.5,-0.5) {};
\draw[thick] (-0.5,-0.5) -- (-0.5,0.5);
\draw[thick] (-0.5,0.5) -- (0.5,0.5);
\draw[thick] (-0.5,-0.5) -- (0.5,-0.5);
\draw[thick] (0.5,0.5) -- (0.5,-0.5);
\draw[thick] (0.5,0.5) -- (1.5,0.5);
\draw[thick] (0.5,0.5) -- (1.5,-0.5);
\draw[thick] (1.5,0.5) -- (1.5,-0.5);
\end{tikzpicture}
\begin{tikzpicture}[scale=0.7]
\filldraw [black]
(-0.5,-0.5) circle (3.5 pt)
(-0.5,0.5) circle (3.5 pt)
(0.5,0.5) circle (3.5 pt)
(0.5,-0.5) circle (3.5 pt)
(1.5,-0.5) circle (3.5 pt)
(1.5,0.5) circle (3.5 pt);
\node [label=below:$C_3$] (C_3) at (0.5,-1.25) {};
\node [label=left:$x$] (x) at (-0.5,0.5) {};
\node [label=left:$y$] (y) at (-0.5,-0.5) {};
\node [label=above:$a$] (a) at (0.5,0.5) {};
\node [label=below:$b$] (b) at (0.5,-0.5) {};
\draw[thick] (-0.5,-0.5) -- (-0.5,0.5);
\draw[thick] (-0.5,0.5) -- (0.5,0.5);
\draw[thick] (-0.5,-0.5) -- (0.5,-0.5);
\draw[thick] (0.5,0.5) -- (0.5,-0.5);
\draw[thick] (0.5,0.5) -- (1.5,0.5);
\draw[thick] (0.5,0.5) -- (1.5,-0.5);
\draw[thick] (1.5,0.5) -- (1.5,-0.5);
\draw[thick] (0.5,-0.5) -- (1.5,-0.5);
\end{tikzpicture}
\begin{tikzpicture}[scale=0.7]
\filldraw [black]
(-0.5,-0.5) circle (3.5 pt)
(-0.5,0.5) circle (3.5 pt)
(0.5,0.5) circle (3.5 pt)
(0.5,-0.5) circle (3.5 pt)
(1.5,-0.5) circle (3.5 pt)
(1.5,0.5) circle (3.5 pt);
\node [label=below:$C_4$] (C_4) at (0.5,-1.25) {};
\node [label=left:$x$] (x) at (-0.5,0.5) {};
\node [label=left:$y$] (y) at (-0.5,-0.5) {};
\node [label=above:$a$] (a) at (0.5,0.5) {};
\node [label=below:$b$] (b) at (0.5,-0.5) {};
\draw[thick] (-0.5,-0.5) -- (-0.5,0.5);
\draw[thick] (-0.5,0.5) -- (0.5,0.5);
\draw[thick] (-0.5,-0.5) -- (0.5,-0.5);
\draw[thick] (0.5,0.5) -- (0.5,-0.5);
\draw[thick] (0.5,0.5) -- (1.5,0.5);
\draw[thick] (0.5,0.5) -- (1.5,-0.5);
\draw[thick] (1.5,0.5) -- (1.5,-0.5);
\draw[thick] (0.5,-0.5) -- (1.5,0.5);
\end{tikzpicture}
\begin{tikzpicture}[scale=0.7]
\filldraw [black]
(-0.5,-0.5) circle (3.5 pt)
(-0.5,0.5) circle (3.5 pt)
(0.5,0.5) circle (3.5 pt)
(0.5,-0.5) circle (3.5 pt)
(1.5,-0.5) circle (3.5 pt)
(1.5,0.5) circle (3.5 pt);
\node [label=below:$C_5$] (C_5) at (0.5,-1.25) {};
\node [label=left:$x$] (x) at (-0.5,0.5) {};
\node [label=left:$y$] (y) at (-0.5,-0.5) {};
\node [label=above:$a$] (a) at (0.5,0.5) {};
\node [label=below:$b$] (b) at (0.5,-0.5) {};
\draw[thick] (-0.5,-0.5) -- (-0.5,0.5);
\draw[thick] (-0.5,0.5) -- (0.5,0.5);
\draw[thick] (-0.5,-0.5) -- (0.5,-0.5);
\draw[thick] (0.5,0.5) -- (0.5,-0.5);
\draw[thick] (0.5,0.5) -- (1.5,0.5);
\draw[thick] (0.5,0.5) -- (1.5,-0.5);
\draw[thick] (1.5,0.5) -- (1.5,-0.5);
\draw[thick] (0.5,-0.5) -- (1.5,0.5);
\draw[thick] (0.5,-0.5) -- (1.5,-0.5);
\end{tikzpicture}
\begin{tikzpicture}[scale=0.7]
\filldraw [black]
(-0.5,-0.5) circle (3.5 pt)
(-0.5,0.5) circle (3.5 pt)
(0.5,0.5) circle (3.5 pt)
(0.5,-0.5) circle (3.5 pt)
(1.5,-0.5) circle (3.5 pt)
(1.5,0.5) circle (3.5 pt)
(2.25,0) circle (3.5 pt);
\node [label=below:$C_6$] (C_6) at (0.75,-1.25) {};
\node [label=left:$x$] (x) at (-0.5,0.5) {};
\node [label=left:$y$] (y) at (-0.5,-0.5) {};
\node [label=above:$a$] (a) at (0.5,0.5) {};
\node [label=below:$b$] (b) at (0.5,-0.5) {};
\draw[thick] (-0.5,-0.5) -- (-0.5,0.5);
\draw[thick] (-0.5,0.5) -- (0.5,0.5);
\draw[thick] (-0.5,-0.5) -- (0.5,-0.5);
\draw[thick] (0.5,0.5) -- (0.5,-0.5);
\draw[thick] (0.5,0.5) -- (1.5,0.5);
\draw[thick] (1.5,0.5) -- (1.5,-0.5);
\draw[thick] (1.5,0.5) -- (2.25,0);
\draw[thick] (1.5,-0.5) -- (2.25,0);
\end{tikzpicture}\\
\begin{tikzpicture}[scale=0.7]
\filldraw [black]
(-0.5,-0.5) circle (3.5 pt)
(-0.5,0.5) circle (3.5 pt)
(0.5,0.5) circle (3.5 pt)
(0.5,-0.5) circle (3.5 pt)
(1.5,-0.5) circle (3.5 pt)
(1.5,0.5) circle (3.5 pt)
(2.25,0) circle (3.5 pt);
\node [label=below:$C_7$] (C_7) at (0.75,-1.25) {};
\node [label=left:$x$] (x) at (-0.5,0.5) {};
\node [label=left:$y$] (y) at (-0.5,-0.5) {};
\node [label=above:$a$] (a) at (0.5,0.5) {};
\node [label=below:$b$] (b) at (0.5,-0.5) {};
\draw[thick] (-0.5,-0.5) -- (-0.5,0.5);
\draw[thick] (-0.5,0.5) -- (0.5,0.5);
\draw[thick] (-0.5,-0.5) -- (0.5,-0.5);
\draw[thick] (0.5,0.5) -- (0.5,-0.5);
\draw[thick] (0.5,0.5) -- (1.5,0.5);
\draw[thick] (1.5,0.5) -- (1.5,-0.5);
\draw[thick] (1.5,0.5) -- (2.25,0);
\draw[thick] (1.5,-0.5) -- (2.25,0);
\draw[thick] (1.5,0.5) -- (0.5,-0.5);
\end{tikzpicture}
\begin{tikzpicture}[scale=0.7]
\filldraw [black]
(-0.5,-0.5) circle (3.5 pt)
(-0.5,0.5) circle (3.5 pt)
(0.5,0.5) circle (3.5 pt)
(0.5,-0.5) circle (3.5 pt)
(1.5,-0.5) circle (3.5 pt)
(1.5,0.5) circle (3.5 pt)
(2.25,0) circle (3.5 pt);
\node [label=below:$C_8$] (C_8) at (0.75,-1.25) {};
\node [label=left:$x$] (x) at (-0.5,0.5) {};
\node [label=left:$y$] (y) at (-0.5,-0.5) {};
\node [label=above:$a$] (a) at (0.5,0.5) {};
\node [label=below:$b$] (b) at (0.5,-0.5) {};
\draw[thick] (-0.5,-0.5) -- (-0.5,0.5);
\draw[thick] (-0.5,0.5) -- (0.5,0.5);
\draw[thick] (-0.5,-0.5) -- (0.5,-0.5);
\draw[thick] (0.5,0.5) -- (0.5,-0.5);
\draw[thick] (0.5,0.5) -- (1.5,0.5);
\draw[thick] (1.5,0.5) -- (1.5,-0.5);
\draw[thick] (1.5,0.5) -- (2.25,0);
\draw[thick] (1.5,-0.5) -- (2.25,0);
\draw[thick] (1.5,-0.5) -- (0.5,-0.5);
\end{tikzpicture}
\begin{tikzpicture}[scale=0.7]
\filldraw [black]
(-0.5,-0.5) circle (3.5 pt)
(-0.5,0.5) circle (3.5 pt)
(0.5,0.5) circle (3.5 pt)
(0.5,-0.5) circle (3.5 pt)
(1.5,-0.5) circle (3.5 pt)
(1.5,0.5) circle (3.5 pt)
(2.25,0) circle (3.5 pt);
\node [label=below:$C_9$] (C_9) at (0.75,-1.25) {};
\node [label=left:$x$] (x) at (-0.5,0.5) {};
\node [label=left:$y$] (y) at (-0.5,-0.5) {};
\node [label=above:$a$] (a) at (0.5,0.5) {};
\node [label=below:$b$] (b) at (0.5,-0.5) {};
\draw[thick] (-0.5,-0.5) -- (-0.5,0.5);
\draw[thick] (-0.5,0.5) -- (0.5,0.5);
\draw[thick] (-0.5,-0.5) -- (0.5,-0.5);
\draw[thick] (0.5,0.5) -- (0.5,-0.5);
\draw[thick] (0.5,0.5) -- (1.5,0.5);
\draw[thick] (1.5,0.5) -- (1.5,-0.5);
\draw[thick] (1.5,0.5) -- (2.25,0);
\draw[thick] (1.5,-0.5) -- (2.25,0);
\draw[thick] (1.5,-0.5) -- (0.5,0.5);
\end{tikzpicture}
\begin{tikzpicture}[scale=0.7]
\filldraw [black]
(-0.5,-0.5) circle (3.5 pt)
(-0.5,0.5) circle (3.5 pt)
(0.5,0.5) circle (3.5 pt)
(0.5,-0.5) circle (3.5 pt)
(1.5,-0.5) circle (3.5 pt)
(1.5,0.5) circle (3.5 pt)
(2.25,0) circle (3.5 pt);
\node [label=below:$C_{10}$] (C_{10}) at (0.75,-1.25) {};
\node [label=left:$x$] (x) at (-0.5,0.5) {};
\node [label=left:$y$] (y) at (-0.5,-0.5) {};
\node [label=above:$a$] (a) at (0.5,0.5) {};
\node [label=below:$b$] (b) at (0.5,-0.5) {};
\draw[thick] (-0.5,-0.5) -- (-0.5,0.5);
\draw[thick] (-0.5,0.5) -- (0.5,0.5);
\draw[thick] (-0.5,-0.5) -- (0.5,-0.5);
\draw[thick] (0.5,0.5) -- (0.5,-0.5);
\draw[thick] (0.5,0.5) -- (1.5,0.5);
\draw[thick] (1.5,0.5) -- (1.5,-0.5);
\draw[thick] (1.5,0.5) -- (2.25,0);
\draw[thick] (1.5,-0.5) -- (2.25,0);
\draw[thick] (1.5,0.5) -- (0.5,0.5);
\draw[thick] (1.5,0.5) -- (0.5,-0.5);
\draw[thick] (1.5,-0.5) -- (0.5,0.5);
\end{tikzpicture}
\begin{tikzpicture}[scale=0.7]
\filldraw [black]
(-0.5,-0.5) circle (3.5 pt)
(-0.5,0.5) circle (3.5 pt)
(0.5,0.5) circle (3.5 pt)
(0.5,-0.5) circle (3.5 pt)
(1.5,-0.5) circle (3.5 pt)
(1.5,0.5) circle (3.5 pt)
(2.25,0) circle (3.5 pt);
\node [label=below:$C_{11}$] (C_{11}) at (0.75,-1.25) {};
\node [label=left:$x$] (x) at (-0.5,0.5) {};
\node [label=left:$y$] (y) at (-0.5,-0.5) {};
\node [label=above:$a$] (a) at (0.5,0.5) {};
\node [label=below:$b$] (b) at (0.5,-0.5) {};
\draw[thick] (-0.5,-0.5) -- (-0.5,0.5);
\draw[thick] (-0.5,0.5) -- (0.5,0.5);
\draw[thick] (-0.5,-0.5) -- (0.5,-0.5);
\draw[thick] (0.5,0.5) -- (0.5,-0.5);
\draw[thick] (0.5,0.5) -- (1.5,0.5);
\draw[thick] (1.5,0.5) -- (1.5,-0.5);
\draw[thick] (1.5,0.5) -- (2.25,0);
\draw[thick] (1.5,-0.5) -- (2.25,0);
\draw[thick] (1.5,0.5) -- (0.5,0.5);
\draw[thick] (1.5,0.5) -- (0.5,-0.5);
\draw[thick] (1.5,-0.5) -- (0.5,-0.5);
\end{tikzpicture}\\
\begin{tikzpicture}[scale=0.7]
\filldraw [black]
(-0.5,-0.5) circle (3.5 pt)
(-0.5,0.5) circle (3.5 pt)
(0.5,0.5) circle (3.5 pt)
(0.5,-0.5) circle (3.5 pt)
(1.5,-0.5) circle (3.5 pt)
(1.5,0.5) circle (3.5 pt)
(2.25,0) circle (3.5 pt);
\node [label=below:$C_{12}$] (C_{12}) at (0.75,-1.25) {};
\node [label=left:$x$] (x) at (-0.5,0.5) {};
\node [label=left:$y$] (y) at (-0.5,-0.5) {};
\node [label=above:$a$] (a) at (0.5,0.5) {};
\node [label=below:$b$] (b) at (0.5,-0.5) {};
\draw[thick] (-0.5,-0.5) -- (-0.5,0.5);
\draw[thick] (-0.5,0.5) -- (0.5,0.5);
\draw[thick] (-0.5,-0.5) -- (0.5,-0.5);
\draw[thick] (0.5,0.5) -- (0.5,-0.5);
\draw[thick] (0.5,0.5) -- (1.5,0.5);
\draw[thick] (1.5,0.5) -- (1.5,-0.5);
\draw[thick] (1.5,0.5) -- (2.25,0);
\draw[thick] (1.5,-0.5) -- (2.25,0);
\draw[thick] (1.5,0.5) -- (0.5,0.5);
\draw[thick] (1.5,-0.5) -- (0.5,-0.5);
\draw[thick] (2.25,0) -- (0.5,0.5);
\end{tikzpicture}
\begin{tikzpicture}[scale=0.7]
\filldraw [black]
(-0.5,-0.5) circle (3.5 pt)
(-0.5,0.5) circle (3.5 pt)
(0.5,0.5) circle (3.5 pt)
(0.5,-0.5) circle (3.5 pt)
(1.5,-0.5) circle (3.5 pt)
(1.5,0.5) circle (3.5 pt)
(2.25,0) circle (3.5 pt);
\node [label=below:$C_{13}$] (C_{13}) at (0.75,-1.25) {};
\node [label=left:$x$] (x) at (-0.5,0.5) {};
\node [label=left:$y$] (y) at (-0.5,-0.5) {};
\node [label=above:$a$] (a) at (0.5,0.5) {};
\node [label=below:$b$] (b) at (0.5,-0.5) {};
\draw[thick] (-0.5,-0.5) -- (-0.5,0.5);
\draw[thick] (-0.5,0.5) -- (0.5,0.5);
\draw[thick] (-0.5,-0.5) -- (0.5,-0.5);
\draw[thick] (0.5,0.5) -- (0.5,-0.5);
\draw[thick] (0.5,0.5) -- (1.5,0.5);
\draw[thick] (1.5,0.5) -- (1.5,-0.5);
\draw[thick] (1.5,0.5) -- (2.25,0);
\draw[thick] (1.5,-0.5) -- (2.25,0);
\draw[thick] (1.5,0.5) -- (0.5,0.5);
\draw[thick] (1.5,-0.5) -- (0.5,0.5);
\draw[thick] (2.25,0) -- (0.5,-0.5);
\end{tikzpicture}
\begin{tikzpicture}[scale=0.7]
\filldraw [black]
(-0.5,-0.5) circle (3.5 pt)
(-0.5,0.5) circle (3.5 pt)
(0.5,0.5) circle (3.5 pt)
(0.5,-0.5) circle (3.5 pt)
(1.5,-0.5) circle (3.5 pt)
(1.5,0.5) circle (3.5 pt)
(2.25,0) circle (3.5 pt);
\node [label=below:$C_{14}$] (C_{14}) at (0.75,-1.25) {};
\node [label=left:$x$] (x) at (-0.5,0.5) {};
\node [label=left:$y$] (y) at (-0.5,-0.5) {};
\node [label=above:$a$] (a) at (0.5,0.5) {};
\node [label=below:$b$] (b) at (0.5,-0.5) {};
\draw[thick] (-0.5,-0.5) -- (-0.5,0.5);
\draw[thick] (-0.5,0.5) -- (0.5,0.5);
\draw[thick] (-0.5,-0.5) -- (0.5,-0.5);
\draw[thick] (0.5,0.5) -- (0.5,-0.5);
\draw[thick] (0.5,0.5) -- (1.5,0.5);
\draw[thick] (1.5,0.5) -- (1.5,-0.5);
\draw[thick] (1.5,0.5) -- (2.25,0);
\draw[thick] (1.5,-0.5) -- (2.25,0);
\draw[thick] (1.5,0.5) -- (0.5,0.5);
\draw[thick] (1.5,-0.5) -- (0.5,0.5);
\draw[thick] (2.25,0) -- (0.5,0.5);
\end{tikzpicture}
\begin{tikzpicture}[scale=0.7]
\filldraw [black]
(-0.5,-0.5) circle (3.5 pt)
(-0.5,0.5) circle (3.5 pt)
(0.5,0.5) circle (3.5 pt)
(0.5,-0.5) circle (3.5 pt)
(1.5,-0.5) circle (3.5 pt)
(1.5,0.5) circle (3.5 pt)
(2.25,0) circle (3.5 pt);
\node [label=below:$C_{15}$] (C_{15}) at (0.75,-1.25) {};
\node [label=left:$x$] (x) at (-0.5,0.5) {};
\node [label=left:$y$] (y) at (-0.5,-0.5) {};
\node [label=above:$a$] (a) at (0.5,0.5) {};
\node [label=below:$b$] (b) at (0.5,-0.5) {};
\draw[thick] (-0.5,-0.5) -- (-0.5,0.5);
\draw[thick] (-0.5,0.5) -- (0.5,0.5);
\draw[thick] (-0.5,-0.5) -- (0.5,-0.5);
\draw[thick] (0.5,0.5) -- (0.5,-0.5);
\draw[thick] (0.5,0.5) -- (1.5,0.5);
\draw[thick] (1.5,0.5) -- (1.5,-0.5);
\draw[thick] (1.5,0.5) -- (2.25,0);
\draw[thick] (1.5,-0.5) -- (2.25,0);
\draw[thick] (1.5,0.5) -- (0.5,0.5);
\draw[thick] (1.5,-0.5) -- (0.5,0.5);
\draw[thick] (1.5,0.5) -- (0.5,-0.5);
\draw[thick] (1.5,-0.5) -- (0.5,-0.5);
\end{tikzpicture}
\begin{tikzpicture}[scale=0.7]
\filldraw [black]
(-0.5,-0.5) circle (3.5 pt)
(-0.5,0.5) circle (3.5 pt)
(0.5,0.5) circle (3.5 pt)
(0.5,-0.5) circle (3.5 pt)
(1.5,-0.5) circle (3.5 pt)
(1.5,0.5) circle (3.5 pt)
(2.25,0) circle (3.5 pt);
\node [label=below:$C_{16}$] (C_{16}) at (0.75,-1.25) {};
\node [label=left:$x$] (x) at (-0.5,0.5) {};
\node [label=left:$y$] (y) at (-0.5,-0.5) {};
\node [label=above:$a$] (a) at (0.5,0.5) {};
\node [label=below:$b$] (b) at (0.5,-0.5) {};
\draw[thick] (-0.5,-0.5) -- (-0.5,0.5);
\draw[thick] (-0.5,0.5) -- (0.5,0.5);
\draw[thick] (-0.5,-0.5) -- (0.5,-0.5);
\draw[thick] (0.5,0.5) -- (0.5,-0.5);
\draw[thick] (0.5,0.5) -- (1.5,0.5);
\draw[thick] (1.5,0.5) -- (1.5,-0.5);
\draw[thick] (1.5,0.5) -- (2.25,0);
\draw[thick] (1.5,-0.5) -- (2.25,0);
\draw[thick] (1.5,0.5) -- (0.5,0.5);
\draw[thick] (1.5,-0.5) -- (0.5,0.5);
\draw[thick] (1.5,0.5) -- (0.5,-0.5);
\draw[thick] (2.25,0) -- (0.5,0.5);
\end{tikzpicture}\\
\begin{tikzpicture}[scale=0.7]
\filldraw [black]
(-0.5,-0.5) circle (3.5 pt)
(-0.5,0.5) circle (3.5 pt)
(0.5,0.5) circle (3.5 pt)
(0.5,-0.5) circle (3.5 pt)
(1.5,-0.5) circle (3.5 pt)
(1.5,0.5) circle (3.5 pt)
(2.25,0) circle (3.5 pt);
\node [label=below:$C_{17}$] (C_{17}) at (0.75,-1.25) {};
\node [label=left:$x$] (x) at (-0.5,0.5) {};
\node [label=left:$y$] (y) at (-0.5,-0.5) {};
\node [label=above:$a$] (a) at (0.5,0.5) {};
\node [label=below:$b$] (b) at (0.5,-0.5) {};
\draw[thick] (-0.5,-0.5) -- (-0.5,0.5);
\draw[thick] (-0.5,0.5) -- (0.5,0.5);
\draw[thick] (-0.5,-0.5) -- (0.5,-0.5);
\draw[thick] (0.5,0.5) -- (0.5,-0.5);
\draw[thick] (0.5,0.5) -- (1.5,0.5);
\draw[thick] (1.5,0.5) -- (1.5,-0.5);
\draw[thick] (1.5,0.5) -- (2.25,0);
\draw[thick] (1.5,-0.5) -- (2.25,0);
\draw[thick] (1.5,0.5) -- (0.5,0.5);
\draw[thick] (1.5,-0.5) -- (0.5,0.5);
\draw[thick] (1.5,0.5) -- (0.5,-0.5);
\draw[thick] (2.25,0) -- (0.5,-0.5);
\end{tikzpicture}
\begin{tikzpicture}[scale=0.7]
\filldraw [black]
(-0.5,-0.5) circle (3.5 pt)
(-0.5,0.5) circle (3.5 pt)
(0.5,0.5) circle (3.5 pt)
(0.5,-0.5) circle (3.5 pt)
(1.5,-0.5) circle (3.5 pt)
(1.5,0.5) circle (3.5 pt)
(2.25,0) circle (3.5 pt);
\node [label=below:$C_{18}$] (C_{18}) at (0.75,-1.25) {};
\node [label=left:$x$] (x) at (-0.5,0.5) {};
\node [label=left:$y$] (y) at (-0.5,-0.5) {};
\node [label=above:$a$] (a) at (0.5,0.5) {};
\node [label=below:$b$] (b) at (0.5,-0.5) {};
\draw[thick] (-0.5,-0.5) -- (-0.5,0.5);
\draw[thick] (-0.5,0.5) -- (0.5,0.5);
\draw[thick] (-0.5,-0.5) -- (0.5,-0.5);
\draw[thick] (0.5,0.5) -- (0.5,-0.5);
\draw[thick] (0.5,0.5) -- (1.5,0.5);
\draw[thick] (1.5,0.5) -- (1.5,-0.5);
\draw[thick] (1.5,0.5) -- (2.25,0);
\draw[thick] (1.5,-0.5) -- (2.25,0);
\draw[thick] (1.5,0.5) -- (0.5,0.5);
\draw[thick] (1.5,-0.5) -- (0.5,-0.5);
\draw[thick] (1.5,0.5) -- (0.5,-0.5);
\draw[thick] (2.25,0) -- (0.5,0.5);
\end{tikzpicture}
\begin{tikzpicture}[scale=0.7]
\filldraw [black]
(-0.5,-0.5) circle (3.5 pt)
(-0.5,0.5) circle (3.5 pt)
(0.5,0.5) circle (3.5 pt)
(0.5,-0.5) circle (3.5 pt)
(1.5,-0.5) circle (3.5 pt)
(1.5,0.5) circle (3.5 pt)
(2.25,0) circle (3.5 pt);
\node [label=below:$C_{19}$] (C_{19}) at (0.75,-1.25) {};
\node [label=left:$x$] (x) at (-0.5,0.5) {};
\node [label=left:$y$] (y) at (-0.5,-0.5) {};
\node [label=above:$a$] (a) at (0.5,0.5) {};
\node [label=below:$b$] (b) at (0.5,-0.5) {};
\draw[thick] (-0.5,-0.5) -- (-0.5,0.5);
\draw[thick] (-0.5,0.5) -- (0.5,0.5);
\draw[thick] (-0.5,-0.5) -- (0.5,-0.5);
\draw[thick] (0.5,0.5) -- (0.5,-0.5);
\draw[thick] (0.5,0.5) -- (1.5,0.5);
\draw[thick] (1.5,0.5) -- (1.5,-0.5);
\draw[thick] (1.5,0.5) -- (2.25,0);
\draw[thick] (1.5,-0.5) -- (2.25,0);
\draw[thick] (1.5,0.5) -- (0.5,0.5);
\draw[thick] (1.5,-0.5) -- (0.5,-0.5);
\draw[thick] (1.5,0.5) -- (0.5,-0.5);
\draw[thick] (2.25,0) -- (0.5,-0.5);
\end{tikzpicture}
\begin{tikzpicture}[scale=0.7]
\filldraw [black]
(-0.5,-0.5) circle (3.5 pt)
(-0.5,0.5) circle (3.5 pt)
(0.5,0.5) circle (3.5 pt)
(0.5,-0.5) circle (3.5 pt)
(1.5,-0.5) circle (3.5 pt)
(1.5,0.5) circle (3.5 pt)
(2.25,0) circle (3.5 pt);
\node [label=below:$C_{20}$] (C_{20}) at (0.75,-1.25) {};
\node [label=left:$x$] (x) at (-0.5,0.5) {};
\node [label=left:$y$] (y) at (-0.5,-0.5) {};
\node [label=above:$a$] (a) at (0.5,0.5) {};
\node [label=below:$b$] (b) at (0.5,-0.5) {};
\draw[thick] (-0.5,-0.5) -- (-0.5,0.5);
\draw[thick] (-0.5,0.5) -- (0.5,0.5);
\draw[thick] (-0.5,-0.5) -- (0.5,-0.5);
\draw[thick] (0.5,0.5) -- (0.5,-0.5);
\draw[thick] (0.5,0.5) -- (1.5,0.5);
\draw[thick] (1.5,0.5) -- (1.5,-0.5);
\draw[thick] (1.5,0.5) -- (2.25,0);
\draw[thick] (1.5,-0.5) -- (2.25,0);
\draw[thick] (1.5,0.5) -- (0.5,0.5);
\draw[thick] (1.5,-0.5) -- (0.5,-0.5);
\draw[thick] (2.25,0) -- (0.5,0.5);
\draw[thick] (2.25,0) -- (0.5,-0.5);
\end{tikzpicture}
\begin{tikzpicture}[scale=0.7]
\filldraw [black]
(-0.5,-0.5) circle (3.5 pt)
(-0.5,0.5) circle (3.5 pt)
(0.5,0.5) circle (3.5 pt)
(0.5,-0.5) circle (3.5 pt)
(1.5,-0.5) circle (3.5 pt)
(1.5,0.5) circle (3.5 pt)
(2.25,0) circle (3.5 pt);
\node [label=below:$C_{21}$] (C_{21}) at (0.75,-1.25) {};
\node [label=left:$x$] (x) at (-0.5,0.5) {};
\node [label=left:$y$] (y) at (-0.5,-0.5) {};
\node [label=above:$a$] (a) at (0.5,0.5) {};
\node [label=below:$b$] (b) at (0.5,-0.5) {};
\draw[thick] (-0.5,-0.5) -- (-0.5,0.5);
\draw[thick] (-0.5,0.5) -- (0.5,0.5);
\draw[thick] (-0.5,-0.5) -- (0.5,-0.5);
\draw[thick] (0.5,0.5) -- (0.5,-0.5);
\draw[thick] (0.5,0.5) -- (1.5,0.5);
\draw[thick] (1.5,0.5) -- (1.5,-0.5);
\draw[thick] (1.5,0.5) -- (2.25,0);
\draw[thick] (1.5,-0.5) -- (2.25,0);
\draw[thick] (1.5,0.5) -- (0.5,0.5);
\draw[thick] (1.5,-0.5) -- (0.5,0.5);
\draw[thick] (2.25,0) -- (0.5,0.5);
\draw[thick] (2.25,0) -- (0.5,-0.5);
\end{tikzpicture}\\
\begin{tikzpicture}[scale=0.7]
\filldraw [black]
(-0.5,-0.5) circle (3.5 pt)
(-0.5,0.5) circle (3.5 pt)
(0.5,0.5) circle (3.5 pt)
(0.5,-0.5) circle (3.5 pt)
(1.5,-0.5) circle (3.5 pt)
(1.5,0.5) circle (3.5 pt)
(2.25,0) circle (3.5 pt);
\node [label=below:$C_{22}$] (C_{22}) at (0.75,-1.25) {};
\node [label=left:$x$] (x) at (-0.5,0.5) {};
\node [label=left:$y$] (y) at (-0.5,-0.5) {};
\node [label=above:$a$] (a) at (0.5,0.5) {};
\node [label=below:$b$] (b) at (0.5,-0.5) {};
\draw[thick] (-0.5,-0.5) -- (-0.5,0.5);
\draw[thick] (-0.5,0.5) -- (0.5,0.5);
\draw[thick] (-0.5,-0.5) -- (0.5,-0.5);
\draw[thick] (0.5,0.5) -- (0.5,-0.5);
\draw[thick] (0.5,0.5) -- (1.5,0.5);
\draw[thick] (1.5,0.5) -- (1.5,-0.5);
\draw[thick] (1.5,0.5) -- (2.25,0);
\draw[thick] (1.5,-0.5) -- (2.25,0);
\draw[thick] (1.5,0.5) -- (0.5,0.5);
\draw[thick] (1.5,0.5) -- (0.5,-0.5);
\draw[thick] (1.5,-0.5) -- (0.5,0.5);
\draw[thick] (1.5,-0.5) -- (0.5,-0.5);
\draw[thick] (2.25,0) -- (0.5,0.5);
\end{tikzpicture}
\begin{tikzpicture}[scale=0.7]
\filldraw [black]
(-0.5,-0.5) circle (3.5 pt)
(-0.5,0.5) circle (3.5 pt)
(0.5,0.5) circle (3.5 pt)
(0.5,-0.5) circle (3.5 pt)
(1.5,-0.5) circle (3.5 pt)
(1.5,0.5) circle (3.5 pt)
(2.25,0) circle (3.5 pt);
\node [label=below:$C_{23}$] (C_{23}) at (0.75,-1.25) {};
\node [label=left:$x$] (x) at (-0.5,0.5) {};
\node [label=left:$y$] (y) at (-0.5,-0.5) {};
\node [label=above:$a$] (a) at (0.5,0.5) {};
\node [label=below:$b$] (b) at (0.5,-0.5) {};
\draw[thick] (-0.5,-0.5) -- (-0.5,0.5);
\draw[thick] (-0.5,0.5) -- (0.5,0.5);
\draw[thick] (-0.5,-0.5) -- (0.5,-0.5);
\draw[thick] (0.5,0.5) -- (0.5,-0.5);
\draw[thick] (0.5,0.5) -- (1.5,0.5);
\draw[thick] (1.5,0.5) -- (1.5,-0.5);
\draw[thick] (1.5,0.5) -- (2.25,0);
\draw[thick] (1.5,-0.5) -- (2.25,0);
\draw[thick] (1.5,0.5) -- (0.5,0.5);
\draw[thick] (1.5,0.5) -- (0.5,-0.5);
\draw[thick] (1.5,-0.5) -- (0.5,0.5);
\draw[thick] (2.25,0) -- (0.5,0.5);
\draw[thick] (2.25,0) -- (0.5,-0.5);
\end{tikzpicture}
\begin{tikzpicture}[scale=0.7]
\filldraw [black]
(-0.5,-0.5) circle (3.5 pt)
(-0.5,0.5) circle (3.5 pt)
(0.5,0.5) circle (3.5 pt)
(0.5,-0.5) circle (3.5 pt)
(1.5,-0.5) circle (3.5 pt)
(1.5,0.5) circle (3.5 pt)
(2.25,0) circle (3.5 pt);
\node [label=below:$C_{24}$] (C_{24}) at (0.75,-1.25) {};
\node [label=left:$x$] (x) at (-0.5,0.5) {};
\node [label=left:$y$] (y) at (-0.5,-0.5) {};
\node [label=above:$a$] (a) at (0.5,0.5) {};
\node [label=below:$b$] (b) at (0.5,-0.5) {};
\draw[thick] (-0.5,-0.5) -- (-0.5,0.5);
\draw[thick] (-0.5,0.5) -- (0.5,0.5);
\draw[thick] (-0.5,-0.5) -- (0.5,-0.5);
\draw[thick] (0.5,0.5) -- (0.5,-0.5);
\draw[thick] (0.5,0.5) -- (1.5,0.5);
\draw[thick] (1.5,0.5) -- (1.5,-0.5);
\draw[thick] (1.5,0.5) -- (2.25,0);
\draw[thick] (1.5,-0.5) -- (2.25,0);
\draw[thick] (1.5,0.5) -- (0.5,0.5);
\draw[thick] (1.5,0.5) -- (0.5,-0.5);
\draw[thick] (1.5,-0.5) -- (0.5,-0.5);
\draw[thick] (2.25,0) -- (0.5,0.5);
\draw[thick] (2.25,0) -- (0.5,-0.5);
\end{tikzpicture}
\begin{tikzpicture}[scale=0.7]
\filldraw [black]
(-0.5,-0.5) circle (3.5 pt)
(-0.5,0.5) circle (3.5 pt)
(0.5,0.5) circle (3.5 pt)
(0.5,-0.5) circle (3.5 pt)
(1.5,-0.5) circle (3.5 pt)
(1.5,0.5) circle (3.5 pt)
(2.25,0) circle (3.5 pt);
\node [label=below:$C_{25}$] (C_{25}) at (0.75,-1.25) {};
\node [label=left:$x$] (x) at (-0.5,0.5) {};
\node [label=left:$y$] (y) at (-0.5,-0.5) {};
\node [label=above:$a$] (a) at (0.5,0.5) {};
\node [label=below:$b$] (b) at (0.5,-0.5) {};
\draw[thick] (-0.5,-0.5) -- (-0.5,0.5);
\draw[thick] (-0.5,0.5) -- (0.5,0.5);
\draw[thick] (-0.5,-0.5) -- (0.5,-0.5);
\draw[thick] (0.5,0.5) -- (0.5,-0.5);
\draw[thick] (0.5,0.5) -- (1.5,0.5);
\draw[thick] (1.5,0.5) -- (1.5,-0.5);
\draw[thick] (1.5,0.5) -- (2.25,0);
\draw[thick] (1.5,-0.5) -- (2.25,0);
\draw[thick] (1.5,0.5) -- (0.5,0.5);
\draw[thick] (1.5,-0.5) -- (0.5,0.5);
\draw[thick] (1.5,-0.5) -- (0.5,-0.5);
\draw[thick] (2.25,0) -- (0.5,0.5);
\draw[thick] (2.25,0) -- (0.5,-0.5);
\end{tikzpicture}
\begin{tikzpicture}[scale=0.7]
\filldraw [black]
(-0.5,-0.5) circle (3.5 pt)
(-0.5,0.5) circle (3.5 pt)
(0.5,0.5) circle (3.5 pt)
(0.5,-0.5) circle (3.5 pt)
(1.5,-0.5) circle (3.5 pt)
(1.5,0.5) circle (3.5 pt)
(2.25,0) circle (3.5 pt);
\node [label=below:$C_{26}$] (C_{26}) at (0.75,-1.25) {};
\node [label=left:$x$] (x) at (-0.5,0.5) {};
\node [label=left:$y$] (y) at (-0.5,-0.5) {};
\node [label=above:$a$] (a) at (0.5,0.5) {};
\node [label=below:$b$] (b) at (0.5,-0.5) {};
\draw[thick] (-0.5,-0.5) -- (-0.5,0.5);
\draw[thick] (-0.5,0.5) -- (0.5,0.5);
\draw[thick] (-0.5,-0.5) -- (0.5,-0.5);
\draw[thick] (0.5,0.5) -- (0.5,-0.5);
\draw[thick] (0.5,0.5) -- (1.5,0.5);
\draw[thick] (1.5,0.5) -- (1.5,-0.5);
\draw[thick] (1.5,0.5) -- (2.25,0);
\draw[thick] (1.5,-0.5) -- (2.25,0);
\draw[thick] (1.5,0.5) -- (0.5,0.5);
\draw[thick] (1.5,-0.5) -- (0.5,0.5);
\draw[thick] (1.5,-0.5) -- (0.5,-0.5);
\draw[thick] (1.5,0.5) -- (0.5,-0.5);
\draw[thick] (2.25,0) -- (0.5,0.5);
\draw[thick] (2.25,0) -- (0.5,-0.5);
\end{tikzpicture}
\caption{All possible cases between $C$ and a triangle in $G$}\label{pic5}
\end{figure}

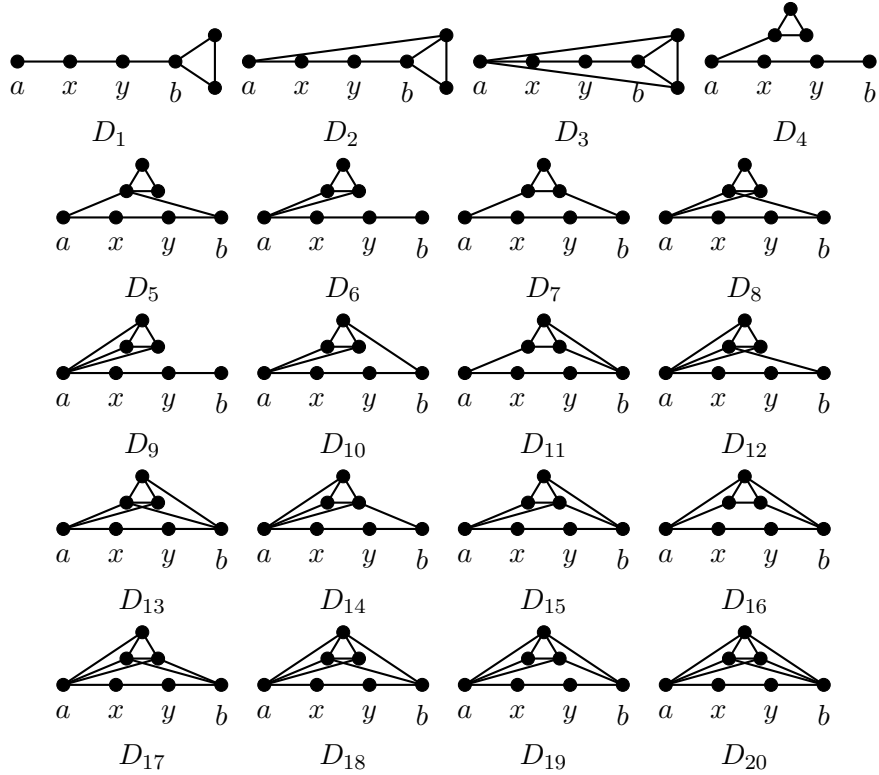
\begin{figure}[htb]
\centering
\begin{tikzpicture}[scale=0.7]
\filldraw [black]
(-1.5,0) circle (3.5 pt)
(-0.5,0) circle (3.5 pt)
(0.5,0) circle (3.5 pt)
(1.5,0) circle (3.5 pt)
(2.25,0.5) circle (3.5 pt)
(2.25,-0.5) circle (3.5 pt);
\node [label=below:$D_1$] (D_1) at (0.25,-0.75) {};
\node [label=below:$x$] (x) at (-0.5,0) {};
\node [label=below:$a$] (a) at (-1.5,0) {};
\node [label=below:$y$] (y) at (0.5,0) {};
\node [label=below:$b$] (b) at (1.5,0) {};
\draw[thick] (-1.5,0) -- (-0.5,0);
\draw[thick] (-0.5,0) -- (0.5,0);
\draw[thick] (0.5,0) -- (1.5,0);
\draw[thick] (1.5,0) -- (2.25,0.5);
\draw[thick] (1.5,0) -- (2.25,-0.5);
\draw[thick] (2.25,0.5) -- (2.25,-0.5);
\end{tikzpicture}
\begin{tikzpicture}[scale=0.7]
\filldraw [black]
(-1.5,0) circle (3.5 pt)
(-0.5,0) circle (3.5 pt)
(0.5,0) circle (3.5 pt)
(1.5,0) circle (3.5 pt)
(2.25,0.5) circle (3.5 pt)
(2.25,-0.5) circle (3.5 pt);
\node [label=below:$D_2$] (D_2) at (0.25,-0.75) {};
\node [label=below:$x$] (x) at (-0.5,0) {};
\node [label=below:$a$] (a) at (-1.5,0) {};
\node [label=below:$y$] (y) at (0.5,0) {};
\node [label=below:$b$] (b) at (1.5,0) {};
\draw[thick] (-1.5,0) -- (-0.5,0);
\draw[thick] (-0.5,0) -- (0.5,0);
\draw[thick] (0.5,0) -- (1.5,0);
\draw[thick] (1.5,0) -- (2.25,0.5);
\draw[thick] (1.5,0) -- (2.25,-0.5);
\draw[thick] (2.25,0.5) -- (2.25,-0.5);
\draw[thick] (2.25,0.5) -- (-1.5,0);
\end{tikzpicture}
\begin{tikzpicture}[scale=0.7]
\filldraw [black]
(-1.5,0) circle (3.5 pt)
(-0.5,0) circle (3.5 pt)
(0.5,0) circle (3.5 pt)
(1.5,0) circle (3.5 pt)
(2.25,0.5) circle (3.5 pt)
(2.25,-0.5) circle (3.5 pt);
\node [label=below:$D_3$] (D_3) at (0.25,-0.75) {};
\node [label=below:$x$] (x) at (-0.5,0) {};
\node [label=below:$a$] (a) at (-1.5,0) {};
\node [label=below:$y$] (y) at (0.5,0) {};
\node [label=below:$b$] (b) at (1.5,0) {};
\draw[thick] (-1.5,0) -- (-0.5,0);
\draw[thick] (-0.5,0) -- (0.5,0);
\draw[thick] (0.5,0) -- (1.5,0);
\draw[thick] (1.5,0) -- (2.25,0.5);
\draw[thick] (1.5,0) -- (2.25,-0.5);
\draw[thick] (2.25,0.5) -- (2.25,-0.5);
\draw[thick] (2.25,0.5) -- (-1.5,0);
\draw[thick] (2.25,-0.5) -- (-1.5,0);
\end{tikzpicture}
\begin{tikzpicture}[scale=0.7]
\filldraw [black]
(-1.5,0) circle (3.5 pt)
(-0.5,0) circle (3.5 pt)
(0.5,0) circle (3.5 pt)
(1.5,0) circle (3.5 pt)
(-0.30,0.5) circle (3.5 pt)
(0.30,0.5) circle (3.5 pt)
(0,1) circle (3.5 pt);
\node [label=below:$D_4$] (D_4) at (0,-0.75) {};
\node [label=below:$x$] (x) at (-0.5,0) {};
\node [label=below:$a$] (a) at (-1.5,0) {};
\node [label=below:$y$] (y) at (0.5,0) {};
\node [label=below:$b$] (b) at (1.5,0) {};
\draw[thick] (-1.5,0) -- (-0.5,0);
\draw[thick] (-0.5,0) -- (0.5,0);
\draw[thick] (0.5,0) -- (1.5,0);
\draw[thick] (-0.30,0.5) -- (0.30,0.5);
\draw[thick] (-0.30,0.5) -- (0,1);
\draw[thick] (0.30,0.5) -- (0,1);
\draw[thick] (-0.30,0.5) -- (-1.5,0);
\end{tikzpicture}\\
\begin{tikzpicture}[scale=0.7]
\filldraw [black]
(-1.5,0) circle (3.5 pt)
(-0.5,0) circle (3.5 pt)
(0.5,0) circle (3.5 pt)
(1.5,0) circle (3.5 pt)
(-0.30,0.5) circle (3.5 pt)
(0.30,0.5) circle (3.5 pt)
(0,1) circle (3.5 pt);
\node [label=below:$D_5$] (D_5) at (0,-0.75) {};
\node [label=below:$x$] (x) at (-0.5,0) {};
\node [label=below:$a$] (a) at (-1.5,0) {};
\node [label=below:$y$] (y) at (0.5,0) {};
\node [label=below:$b$] (b) at (1.5,0) {};
\draw[thick] (-1.5,0) -- (-0.5,0);
\draw[thick] (-0.5,0) -- (0.5,0);
\draw[thick] (0.5,0) -- (1.5,0);
\draw[thick] (-0.30,0.5) -- (0.30,0.5);
\draw[thick] (-0.30,0.5) -- (0,1);
\draw[thick] (0.30,0.5) -- (0,1);
\draw[thick] (-0.30,0.5) -- (-1.5,0);
\draw[thick] (-0.30,0.5) -- (1.5,0);
\end{tikzpicture}
\begin{tikzpicture}[scale=0.7]
\filldraw [black]
(-1.5,0) circle (3.5 pt)
(-0.5,0) circle (3.5 pt)
(0.5,0) circle (3.5 pt)
(1.5,0) circle (3.5 pt)
(-0.30,0.5) circle (3.5 pt)
(0.30,0.5) circle (3.5 pt)
(0,1) circle (3.5 pt);
\node [label=below:$D_6$] (D_6) at (0,-0.75) {};
\node [label=below:$x$] (x) at (-0.5,0) {};
\node [label=below:$a$] (a) at (-1.5,0) {};
\node [label=below:$y$] (y) at (0.5,0) {};
\node [label=below:$b$] (b) at (1.5,0) {};
\draw[thick] (-1.5,0) -- (-0.5,0);
\draw[thick] (-0.5,0) -- (0.5,0);
\draw[thick] (0.5,0) -- (1.5,0);
\draw[thick] (-0.30,0.5) -- (0.30,0.5);
\draw[thick] (-0.30,0.5) -- (0,1);
\draw[thick] (0.30,0.5) -- (0,1);
\draw[thick] (-0.30,0.5) -- (-1.5,0);
\draw[thick] (0.30,0.5) -- (-1.5,0);
\end{tikzpicture}
\begin{tikzpicture}[scale=0.7]
\filldraw [black]
(-1.5,0) circle (3.5 pt)
(-0.5,0) circle (3.5 pt)
(0.5,0) circle (3.5 pt)
(1.5,0) circle (3.5 pt)
(-0.30,0.5) circle (3.5 pt)
(0.30,0.5) circle (3.5 pt)
(0,1) circle (3.5 pt);
\node [label=below:$D_7$] (D_7) at (0,-0.75) {};
\node [label=below:$x$] (x) at (-0.5,0) {};
\node [label=below:$a$] (a) at (-1.5,0) {};
\node [label=below:$y$] (y) at (0.5,0) {};
\node [label=below:$b$] (b) at (1.5,0) {};
\draw[thick] (-1.5,0) -- (-0.5,0);
\draw[thick] (-0.5,0) -- (0.5,0);
\draw[thick] (0.5,0) -- (1.5,0);
\draw[thick] (-0.30,0.5) -- (0.30,0.5);
\draw[thick] (-0.30,0.5) -- (0,1);
\draw[thick] (0.30,0.5) -- (0,1);
\draw[thick] (-0.30,0.5) -- (-1.5,0);
\draw[thick] (0.30,0.5) -- (1.5,0);
\end{tikzpicture}
\begin{tikzpicture}[scale=0.7]
\filldraw [black]
(-1.5,0) circle (3.5 pt)
(-0.5,0) circle (3.5 pt)
(0.5,0) circle (3.5 pt)
(1.5,0) circle (3.5 pt)
(-0.30,0.5) circle (3.5 pt)
(0.30,0.5) circle (3.5 pt)
(0,1) circle (3.5 pt);
\node [label=below:$D_8$] (D_8) at (0,-0.75) {};
\node [label=below:$x$] (x) at (-0.5,0) {};
\node [label=below:$a$] (a) at (-1.5,0) {};
\node [label=below:$y$] (y) at (0.5,0) {};
\node [label=below:$b$] (b) at (1.5,0) {};
\draw[thick] (-1.5,0) -- (-0.5,0);
\draw[thick] (-0.5,0) -- (0.5,0);
\draw[thick] (0.5,0) -- (1.5,0);
\draw[thick] (-0.30,0.5) -- (0.30,0.5);
\draw[thick] (-0.30,0.5) -- (0,1);
\draw[thick] (0.30,0.5) -- (0,1);
\draw[thick] (-0.30,0.5) -- (-1.5,0);
\draw[thick] (-0.30,0.5) -- (1.5,0);
\draw[thick] (0.30,0.5) -- (-1.5,0);
\end{tikzpicture}\\
\begin{tikzpicture}[scale=0.7]
\filldraw [black]
(-1.5,0) circle (3.5 pt)
(-0.5,0) circle (3.5 pt)
(0.5,0) circle (3.5 pt)
(1.5,0) circle (3.5 pt)
(-0.30,0.5) circle (3.5 pt)
(0.30,0.5) circle (3.5 pt)
(0,1) circle (3.5 pt);
\node [label=below:$D_9$] (D_9) at (0,-0.75) {};
\node [label=below:$x$] (x) at (-0.5,0) {};
\node [label=below:$a$] (a) at (-1.5,0) {};
\node [label=below:$y$] (y) at (0.5,0) {};
\node [label=below:$b$] (b) at (1.5,0) {};
\draw[thick] (-1.5,0) -- (-0.5,0);
\draw[thick] (-0.5,0) -- (0.5,0);
\draw[thick] (0.5,0) -- (1.5,0);
\draw[thick] (-0.30,0.5) -- (0.30,0.5);
\draw[thick] (-0.30,0.5) -- (0,1);
\draw[thick] (0.30,0.5) -- (0,1);
\draw[thick] (-0.30,0.5) -- (-1.5,0);
\draw[thick] (0,1) -- (-1.5,0);
\draw[thick] (0.30,0.5) -- (-1.5,0);
\end{tikzpicture}
\begin{tikzpicture}[scale=0.7]
\filldraw [black]
(-1.5,0) circle (3.5 pt)
(-0.5,0) circle (3.5 pt)
(0.5,0) circle (3.5 pt)
(1.5,0) circle (3.5 pt)
(-0.30,0.5) circle (3.5 pt)
(0.30,0.5) circle (3.5 pt)
(0,1) circle (3.5 pt);
\node [label=below:$D_{10}$] (D_{10}) at (0,-0.75) {};
\node [label=below:$x$] (x) at (-0.5,0) {};
\node [label=below:$a$] (a) at (-1.5,0) {};
\node [label=below:$y$] (y) at (0.5,0) {};
\node [label=below:$b$] (b) at (1.5,0) {};
\draw[thick] (-1.5,0) -- (-0.5,0);
\draw[thick] (-0.5,0) -- (0.5,0);
\draw[thick] (0.5,0) -- (1.5,0);
\draw[thick] (-0.30,0.5) -- (0.30,0.5);
\draw[thick] (-0.30,0.5) -- (0,1);
\draw[thick] (0.30,0.5) -- (0,1);
\draw[thick] (-0.30,0.5) -- (-1.5,0);
\draw[thick] (0,1) -- (1.5,0);
\draw[thick] (0.30,0.5) -- (-1.5,0);
\end{tikzpicture}
\begin{tikzpicture}[scale=0.7]
\filldraw [black]
(-1.5,0) circle (3.5 pt)
(-0.5,0) circle (3.5 pt)
(0.5,0) circle (3.5 pt)
(1.5,0) circle (3.5 pt)
(-0.30,0.5) circle (3.5 pt)
(0.30,0.5) circle (3.5 pt)
(0,1) circle (3.5 pt);
\node [label=below:$D_{11}$] (D_{11}) at (0,-0.75) {};
\node [label=below:$x$] (x) at (-0.5,0) {};
\node [label=below:$a$] (a) at (-1.5,0) {};
\node [label=below:$y$] (y) at (0.5,0) {};
\node [label=below:$b$] (b) at (1.5,0) {};
\draw[thick] (-1.5,0) -- (-0.5,0);
\draw[thick] (-0.5,0) -- (0.5,0);
\draw[thick] (0.5,0) -- (1.5,0);
\draw[thick] (-0.30,0.5) -- (0.30,0.5);
\draw[thick] (-0.30,0.5) -- (0,1);
\draw[thick] (0.30,0.5) -- (0,1);
\draw[thick] (-0.30,0.5) -- (-1.5,0);
\draw[thick] (0,1) -- (1.5,0);
\draw[thick] (0.30,0.5) -- (1.5,0);
\end{tikzpicture}
\begin{tikzpicture}[scale=0.7]
\filldraw [black]
(-1.5,0) circle (3.5 pt)
(-0.5,0) circle (3.5 pt)
(0.5,0) circle (3.5 pt)
(1.5,0) circle (3.5 pt)
(-0.30,0.5) circle (3.5 pt)
(0.30,0.5) circle (3.5 pt)
(0,1) circle (3.5 pt);
\node [label=below:$D_{12}$] (D_{12}) at (0,-0.75) {};
\node [label=below:$x$] (x) at (-0.5,0) {};
\node [label=below:$a$] (a) at (-1.5,0) {};
\node [label=below:$y$] (y) at (0.5,0) {};
\node [label=below:$b$] (b) at (1.5,0) {};
\draw[thick] (-1.5,0) -- (-0.5,0);
\draw[thick] (-0.5,0) -- (0.5,0);
\draw[thick] (0.5,0) -- (1.5,0);
\draw[thick] (-0.30,0.5) -- (0.30,0.5);
\draw[thick] (-0.30,0.5) -- (0,1);
\draw[thick] (0.30,0.5) -- (0,1);
\draw[thick] (-0.30,0.5) -- (-1.5,0);
\draw[thick] (-0.30,0.5) -- (1.5,0);
\draw[thick] (0.30,0.5) -- (-1.5,0);
\draw[thick] (0,1) -- (-1.5,0);
\end{tikzpicture}\\
\begin{tikzpicture}[scale=0.7]
\filldraw [black]
(-1.5,0) circle (3.5 pt)
(-0.5,0) circle (3.5 pt)
(0.5,0) circle (3.5 pt)
(1.5,0) circle (3.5 pt)
(-0.30,0.5) circle (3.5 pt)
(0.30,0.5) circle (3.5 pt)
(0,1) circle (3.5 pt);
\node [label=below:$D_{13}$] (D_{13}) at (0,-0.75) {};
\node [label=below:$x$] (x) at (-0.5,0) {};
\node [label=below:$a$] (a) at (-1.5,0) {};
\node [label=below:$y$] (y) at (0.5,0) {};
\node [label=below:$b$] (b) at (1.5,0) {};
\draw[thick] (-1.5,0) -- (-0.5,0);
\draw[thick] (-0.5,0) -- (0.5,0);
\draw[thick] (0.5,0) -- (1.5,0);
\draw[thick] (-0.30,0.5) -- (0.30,0.5);
\draw[thick] (-0.30,0.5) -- (0,1);
\draw[thick] (0.30,0.5) -- (0,1);
\draw[thick] (-0.30,0.5) -- (-1.5,0);
\draw[thick] (-0.30,0.5) -- (1.5,0);
\draw[thick] (0.30,0.5) -- (-1.5,0);
\draw[thick] (0,1) -- (1.5,0);
\end{tikzpicture}
\begin{tikzpicture}[scale=0.7]
\filldraw [black]
(-1.5,0) circle (3.5 pt)
(-0.5,0) circle (3.5 pt)
(0.5,0) circle (3.5 pt)
(1.5,0) circle (3.5 pt)
(-0.30,0.5) circle (3.5 pt)
(0.30,0.5) circle (3.5 pt)
(0,1) circle (3.5 pt);
\node [label=below:$D_{14}$] (D_{14}) at (0,-0.75) {};
\node [label=below:$x$] (x) at (-0.5,0) {};
\node [label=below:$a$] (a) at (-1.5,0) {};
\node [label=below:$y$] (y) at (0.5,0) {};
\node [label=below:$b$] (b) at (1.5,0) {};
\draw[thick] (-1.5,0) -- (-0.5,0);
\draw[thick] (-0.5,0) -- (0.5,0);
\draw[thick] (0.5,0) -- (1.5,0);
\draw[thick] (-0.30,0.5) -- (0.30,0.5);
\draw[thick] (-0.30,0.5) -- (0,1);
\draw[thick] (0.30,0.5) -- (0,1);
\draw[thick] (-0.30,0.5) -- (-1.5,0);
\draw[thick] (0.30,0.5) -- (-1.5,0);
\draw[thick] (0.30,0.5) -- (1.5,0);
\draw[thick] (0,1) -- (-1.5,0);
\end{tikzpicture}
\begin{tikzpicture}[scale=0.7]
\filldraw [black]
(-1.5,0) circle (3.5 pt)
(-0.5,0) circle (3.5 pt)
(0.5,0) circle (3.5 pt)
(1.5,0) circle (3.5 pt)
(-0.30,0.5) circle (3.5 pt)
(0.30,0.5) circle (3.5 pt)
(0,1) circle (3.5 pt);
\node [label=below:$D_{15}$] (D_{15}) at (0,-0.75) {};
\node [label=below:$x$] (x) at (-0.5,0) {};
\node [label=below:$a$] (a) at (-1.5,0) {};
\node [label=below:$y$] (y) at (0.5,0) {};
\node [label=below:$b$] (b) at (1.5,0) {};
\draw[thick] (-1.5,0) -- (-0.5,0);
\draw[thick] (-0.5,0) -- (0.5,0);
\draw[thick] (0.5,0) -- (1.5,0);
\draw[thick] (-0.30,0.5) -- (0.30,0.5);
\draw[thick] (-0.30,0.5) -- (0,1);
\draw[thick] (0.30,0.5) -- (0,1);
\draw[thick] (-0.30,0.5) -- (-1.5,0);
\draw[thick] (0.30,0.5) -- (-1.5,0);
\draw[thick] (0.30,0.5) -- (1.5,0);
\draw[thick] (0,1) -- (1.5,0);
\end{tikzpicture}
\begin{tikzpicture}[scale=0.7]
\filldraw [black]
(-1.5,0) circle (3.5 pt)
(-0.5,0) circle (3.5 pt)
(0.5,0) circle (3.5 pt)
(1.5,0) circle (3.5 pt)
(-0.30,0.5) circle (3.5 pt)
(0.30,0.5) circle (3.5 pt)
(0,1) circle (3.5 pt);
\node [label=below:$D_{16}$] (D_{16}) at (0,-0.75) {};
\node [label=below:$x$] (x) at (-0.5,0) {};
\node [label=below:$a$] (a) at (-1.5,0) {};
\node [label=below:$y$] (y) at (0.5,0) {};
\node [label=below:$b$] (b) at (1.5,0) {};
\draw[thick] (-1.5,0) -- (-0.5,0);
\draw[thick] (-0.5,0) -- (0.5,0);
\draw[thick] (0.5,0) -- (1.5,0);
\draw[thick] (-0.30,0.5) -- (0.30,0.5);
\draw[thick] (-0.30,0.5) -- (0,1);
\draw[thick] (0.30,0.5) -- (0,1);
\draw[thick] (-0.30,0.5) -- (-1.5,0);
\draw[thick] (0.30,0.5) -- (1.5,0);
\draw[thick] (0,1) -- (1.5,0);
\draw[thick] (0,1) -- (-1.5,0);
\end{tikzpicture}\\
\begin{tikzpicture}[scale=0.7]
\filldraw [black]
(-1.5,0) circle (3.5 pt)
(-0.5,0) circle (3.5 pt)
(0.5,0) circle (3.5 pt)
(1.5,0) circle (3.5 pt)
(-0.30,0.5) circle (3.5 pt)
(0.30,0.5) circle (3.5 pt)
(0,1) circle (3.5 pt);
\node [label=below:$D_{17}$] (D_{17}) at (0,-0.75) {};
\node [label=below:$x$] (x) at (-0.5,0) {};
\node [label=below:$a$] (a) at (-1.5,0) {};
\node [label=below:$y$] (y) at (0.5,0) {};
\node [label=below:$b$] (b) at (1.5,0) {};
\draw[thick] (-1.5,0) -- (-0.5,0);
\draw[thick] (-0.5,0) -- (0.5,0);
\draw[thick] (0.5,0) -- (1.5,0);
\draw[thick] (-0.30,0.5) -- (0.30,0.5);
\draw[thick] (-0.30,0.5) -- (0,1);
\draw[thick] (0.30,0.5) -- (0,1);
\draw[thick] (-0.30,0.5) -- (-1.5,0);
\draw[thick] (-0.30,0.5) -- (1.5,0);
\draw[thick] (0.30,0.5) -- (-1.5,0);
\draw[thick] (0.30,0.5) -- (1.5,0);
\draw[thick] (0,1) -- (-1.5,0);
\end{tikzpicture}
\begin{tikzpicture}[scale=0.7]
\filldraw [black]
(-1.5,0) circle (3.5 pt)
(-0.5,0) circle (3.5 pt)
(0.5,0) circle (3.5 pt)
(1.5,0) circle (3.5 pt)
(-0.30,0.5) circle (3.5 pt)
(0.30,0.5) circle (3.5 pt)
(0,1) circle (3.5 pt);
\node [label=below:$D_{18}$] (D_{18}) at (0,-0.75) {};
\node [label=below:$x$] (x) at (-0.5,0) {};
\node [label=below:$a$] (a) at (-1.5,0) {};
\node [label=below:$y$] (y) at (0.5,0) {};
\node [label=below:$b$] (b) at (1.5,0) {};
\draw[thick] (-1.5,0) -- (-0.5,0);
\draw[thick] (-0.5,0) -- (0.5,0);
\draw[thick] (0.5,0) -- (1.5,0);
\draw[thick] (-0.30,0.5) -- (0.30,0.5);
\draw[thick] (-0.30,0.5) -- (0,1);
\draw[thick] (0.30,0.5) -- (0,1);
\draw[thick] (-0.30,0.5) -- (-1.5,0);
\draw[thick] (-0.30,0.5) -- (1.5,0);
\draw[thick] (0.30,0.5) -- (-1.5,0);
\draw[thick] (0,1) -- (-1.5,0);
\draw[thick] (0,1) -- (1.5,0);
\end{tikzpicture}
\begin{tikzpicture}[scale=0.7]
\filldraw [black]
(-1.5,0) circle (3.5 pt)
(-0.5,0) circle (3.5 pt)
(0.5,0) circle (3.5 pt)
(1.5,0) circle (3.5 pt)
(-0.30,0.5) circle (3.5 pt)
(0.30,0.5) circle (3.5 pt)
(0,1) circle (3.5 pt);
\node [label=below:$D_{19}$] (D_{19}) at (0,-0.75) {};
\node [label=below:$x$] (x) at (-0.5,0) {};
\node [label=below:$a$] (a) at (-1.5,0) {};
\node [label=below:$y$] (y) at (0.5,0) {};
\node [label=below:$b$] (b) at (1.5,0) {};
\draw[thick] (-1.5,0) -- (-0.5,0);
\draw[thick] (-0.5,0) -- (0.5,0);
\draw[thick] (0.5,0) -- (1.5,0);
\draw[thick] (-0.30,0.5) -- (0.30,0.5);
\draw[thick] (-0.30,0.5) -- (0,1);
\draw[thick] (0.30,0.5) -- (0,1);
\draw[thick] (-0.30,0.5) -- (-1.5,0);
\draw[thick] (0.30,0.5) -- (-1.5,0);
\draw[thick] (0.30,0.5) -- (1.5,0);
\draw[thick] (0,1) -- (-1.5,0);
\draw[thick] (0,1) -- (1.5,0);
\end{tikzpicture}
\begin{tikzpicture}[scale=0.7]
\filldraw [black]
(-1.5,0) circle (3.5 pt)
(-0.5,0) circle (3.5 pt)
(0.5,0) circle (3.5 pt)
(1.5,0) circle (3.5 pt)
(-0.30,0.5) circle (3.5 pt)
(0.30,0.5) circle (3.5 pt)
(0,1) circle (3.5 pt);
\node [label=below:$D_{20}$] (D_{20}) at (0,-0.75) {};
\node [label=below:$x$] (x) at (-0.5,0) {};
\node [label=below:$a$] (a) at (-1.5,0) {};
\node [label=below:$y$] (y) at (0.5,0) {};
\node [label=below:$b$] (b) at (1.5,0) {};
\draw[thick] (-1.5,0) -- (-0.5,0);
\draw[thick] (-0.5,0) -- (0.5,0);
\draw[thick] (0.5,0) -- (1.5,0);
\draw[thick] (-0.30,0.5) -- (0.30,0.5);
\draw[thick] (-0.30,0.5) -- (0,1);
\draw[thick] (0.30,0.5) -- (0,1);
\draw[thick] (-0.30,0.5) -- (-1.5,0);
\draw[thick] (0.30,0.5) -- (-1.5,0);
\draw[thick] (0.30,0.5) -- (1.5,0);
\draw[thick] (-0.30,0.5) -- (1.5,0);
\draw[thick] (0,1) -- (-1.5,0);
\draw[thick] (0,1) -- (1.5,0);
\end{tikzpicture}
\caption{All possible cases between $D$ and a triangle in $G$}\label{pic6}
\end{figure}
\end{proof}

Now, suppose the graph $G$ is cospectral with friendship graph $F_n$. We study the case that, two vertices of degree 2 in $G$ are not adjacent. In this case, with one more condition we can prove that $G$ is isomorphic to $F_n$.

\begin{lem}\label{lxy}
Let $\{x,y\}$ be two vertices of degree $2$ in graph $G$, where these vertices are not adjacent. Then $x$ and $y$ does not have common neighbours.
\end{lem}
\begin{proof}
By contrary, suppose $x$ and $y$ have common neighbours, say $\{a,b\}$. So, we have the graph of Figure \ref{pic7} as a subgraph of $G$. Suppose the adjacency matrix of $G$ is $A(G)$ and, the first, second, third and fourth rows and columns of $A(G)$ are labelled by vertices $x, y, a$ and $b$, respectively. The two first rows of $A(G)$ are identical, since $d_G(x)=d_G(y)=2$ and they are not adjacent in $G$. Therefore, the dimension of the null space of $A(G)$ is greater than zero. So, $0$ is an eigenvalue of $A(G)$ and it is contradiction with cospectrality of $G$ and $F_n$. This completes the proof.
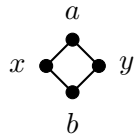
\begin{figure}[htb]
\centering
\begin{tikzpicture}[scale=0.7]
\filldraw [black]
(-0.5,0) circle (3.5 pt)
(0,0.5) circle (3.5 pt)
(0,-0.5) circle (3.5 pt)
(0.5,0) circle (3.5 pt);
\node [label=left:$x$] (x) at (-0.5,0) {};
\node [label=above:$a$] (a) at (0,0.5) {};
\node [label=below:$b$] (b) at (0,-0.5) {};
\node [label=right:$y$] (y) at (0.5,0) {};
\draw[thick] (-0.5,0) -- (0,0.5);
\draw[thick] (-0.5,0) -- (0,-0.5);
\draw[thick] (0,0.5) -- (0.5,0);
\draw[thick] (0,-0.5) -- (0.5,0);
\end{tikzpicture}
\caption{Vertices $\{x,y\}$ are adjacent to $\{a,b\}$ in $G$}\label{pic7}
\end{figure}
\end{proof}
It is known that the Kronecker product of paths $P_2$ and $P_3$, $P_2\times P_3$, is two cycles $C_4$ that has a common edge.

\begin{thm}\label{llxy}
Let $\{x,y\}$ be two non-adjacent vertices of degree 2 in $G$ and $G$ is $P_2\times P_3$-free. If the vertices $x$ and $y$ have at least one common neighbour vertex, then $G$ is isomorphic to $F_n$.
\end{thm}
\begin{proof}
Suppose the common neighbour vertex of two vertices $x$ and $y$ in $G$ is $z$. Also, suppose $x$ and $y$ are adjacent to $a$ and $b$, respectively. By Lemma \ref{lxy}, we can assume that $a\neq b$. So, we have the path $P_5$ with vertices $\{a, x, z, y, b\}$ as a subgraph of $G$. All possible induced subgraph that can be obtained from this $P_5$ are listed in Figure \ref{pic8}. The graphs $E_4, E_5$ and $E_6$ have two negative eigenvalues less than $-1$, so they can not happen in $G$. The vertex $a$ in $E_2$ and $E_3$ must be join to an other vertex in $G$, say $t$. All possible cases for these two graphs with this new edge, are shown in Figure \ref{pic9}. All of them are forbidden subgraph of $G$. So, $E_2$ and $E_3$ can not happen in $G$. Therefore, the only case that can happen is $E_1$. Suppose the vertex $a$ in $E_1$ is adjacent to vertex $t$ of $G$, since the degree of $a$ can not be 1. Now, the vertex $t$ must be adjacent to some vertices of the set $\{z,b\}$, since $G$ is $P_6$-free. It can not be adjacent to the both of $z$ and $b$, since we do not have induced subgraph $P_2\times P_3$. Also, $t$ only is not adjacent to the vertex $z$, since its second largest eigenvalues are greater than 1. The only remaining case is that $t$ be adjacent only to $b$. In this case we have an induced $C_6$ in $G$. If $d_G(b)=2$, then by Lemma \ref{Tx}, $G$ must be isomorphic to $F_n$ and, nothing remain to prove. So, we must show that $d_G(b)$ can not be greater than $2$. But, if $d_G(b)>2$ and $b$ is adjacent to the vertex $f$ of $G$, by interlacing theorem, the vertex $f$ must be adjacent to some vertices of the set $\{z, t, a\}$. But, all the resulted graphs are forbidden in $G$. So, the proof is completed.
\begin{figure}[htb]
\centering
\begin{tikzpicture}[scale=0.7]
\filldraw [black]
(-1,-1) circle (3.5 pt)
(-1,0) circle (3.5 pt)
(0,0) circle (3.5 pt)
(1,-1) circle (3.5 pt)
(1,0) circle (3.5 pt);
\node [label=below:$E_1$] (E_1) at (0,-0.85) {};
\node [label=left:$a$] (a) at (-1,-1) {};
\node [label=above:$x$] (x) at (-1,0) {};
\node [label=above:$z$] (z) at (0,0) {};
\node [label=above:$y$] (y) at (1,0) {};
\node [label=right:$b$] (b) at (1,-1) {};
\draw[thick] (-1,-1) -- (-1,0);
\draw[thick] (-1,0) -- (0,0);
\draw[thick] (0,0) -- (1,0);
\draw[thick] (1,0) -- (1,-1);
\end{tikzpicture}
\begin{tikzpicture}[scale=0.7]
\filldraw [black]
(-1,-1) circle (3.5 pt)
(-1,0) circle (3.5 pt)
(0,0) circle (3.5 pt)
(1,-1) circle (3.5 pt)
(1,0) circle (3.5 pt);
\node [label=below:$E_2$] (E_2) at (0,-0.85) {};
\node [label=left:$a$] (a) at (-1,-1) {};
\node [label=above:$x$] (x) at (-1,0) {};
\node [label=above:$z$] (z) at (0,0) {};
\node [label=above:$y$] (y) at (1,0) {};
\node [label=right:$b$] (b) at (1,-1) {};
\draw[thick] (-1,-1) -- (-1,0);
\draw[thick] (-1,0) -- (0,0);
\draw[thick] (0,0) -- (1,0);
\draw[thick] (1,0) -- (1,-1);
\draw[thick] (-1,-1) -- (0,0);
\draw[thick] (1,-1) -- (0,0);
\end{tikzpicture}
\begin{tikzpicture}[scale=0.7]
\filldraw [black]
(-1,-1) circle (3.5 pt)
(-1,0) circle (3.5 pt)
(0,0) circle (3.5 pt)
(1,-1) circle (3.5 pt)
(1,0) circle (3.5 pt);
\node [label=below:$E_3$] (E_3) at (0,-0.85) {};
\node [label=left:$a$] (a) at (-1,-1) {};
\node [label=above:$x$] (x) at (-1,0) {};
\node [label=above:$z$] (z) at (0,0) {};
\node [label=above:$y$] (y) at (1,0) {};
\node [label=right:$b$] (b) at (1,-1) {};
\draw[thick] (-1,-1) -- (-1,0);
\draw[thick] (-1,0) -- (0,0);
\draw[thick] (0,0) -- (1,0);
\draw[thick] (1,0) -- (1,-1);
\draw[thick] (1,-1) -- (0,0);
\end{tikzpicture}
\begin{tikzpicture}[scale=0.7]
\filldraw [black]
(-1,-1) circle (3.5 pt)
(-1,0) circle (3.5 pt)
(0,0) circle (3.5 pt)
(1,-1) circle (3.5 pt)
(1,0) circle (3.5 pt);
\node [label=below:$E_4$] (E_4) at (0,-0.85) {};
\node [label=left:$a$] (a) at (-1,-1) {};
\node [label=above:$x$] (x) at (-1,0) {};
\node [label=above:$z$] (z) at (0,0) {};
\node [label=above:$y$] (y) at (1,0) {};
\node [label=right:$b$] (b) at (1,-1) {};
\draw[thick] (-1,-1) -- (-1,0);
\draw[thick] (-1,0) -- (0,0);
\draw[thick] (0,0) -- (1,0);
\draw[thick] (1,0) -- (1,-1);
\draw[thick] (1,-1) -- (-1,-1);
\end{tikzpicture}
\begin{tikzpicture}[scale=0.7]
\filldraw [black]
(-1,-1) circle (3.5 pt)
(-1,0) circle (3.5 pt)
(0,0) circle (3.5 pt)
(1,-1) circle (3.5 pt)
(1,0) circle (3.5 pt);
\node [label=below:$E_5$] (E_5) at (0,-0.85) {};
\node [label=left:$a$] (a) at (-1,-1) {};
\node [label=above:$x$] (x) at (-1,0) {};
\node [label=above:$z$] (z) at (0,0) {};
\node [label=above:$y$] (y) at (1,0) {};
\node [label=right:$b$] (b) at (1,-1) {};
\draw[thick] (-1,-1) -- (-1,0);
\draw[thick] (-1,0) -- (0,0);
\draw[thick] (0,0) -- (1,0);
\draw[thick] (1,0) -- (1,-1);
\draw[thick] (1,-1) -- (-1,-1);
\draw[thick] (1,-1) -- (0,0);
\end{tikzpicture}
\begin{tikzpicture}[scale=0.7]
\filldraw [black]
(-1,-1) circle (3.5 pt)
(-1,0) circle (3.5 pt)
(0,0) circle (3.5 pt)
(1,-1) circle (3.5 pt)
(1,0) circle (3.5 pt);
\node [label=below:$E_6$] (E_6) at (0,-0.85) {};
\node [label=left:$a$] (a) at (-1,-1) {};
\node [label=above:$x$] (x) at (-1,0) {};
\node [label=above:$z$] (z) at (0,0) {};
\node [label=above:$y$] (y) at (1,0) {};
\node [label=right:$b$] (b) at (1,-1) {};
\draw[thick] (-1,-1) -- (-1,0);
\draw[thick] (-1,0) -- (0,0);
\draw[thick] (0,0) -- (1,0);
\draw[thick] (1,0) -- (1,-1);
\draw[thick] (1,-1) -- (-1,-1);
\draw[thick] (1,-1) -- (0,0);
\draw[thick] (-1,-1) -- (0,0);
\end{tikzpicture}
\caption{All induced subgraphs of $P_5$ of Lemma \ref{llxy}}\label{pic8}
\end{figure}
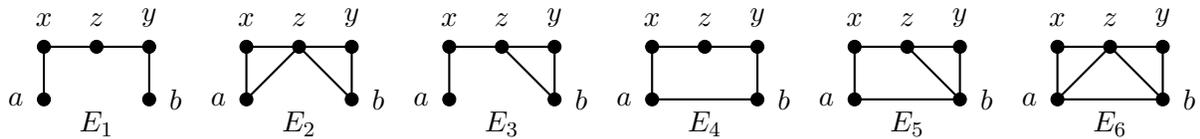
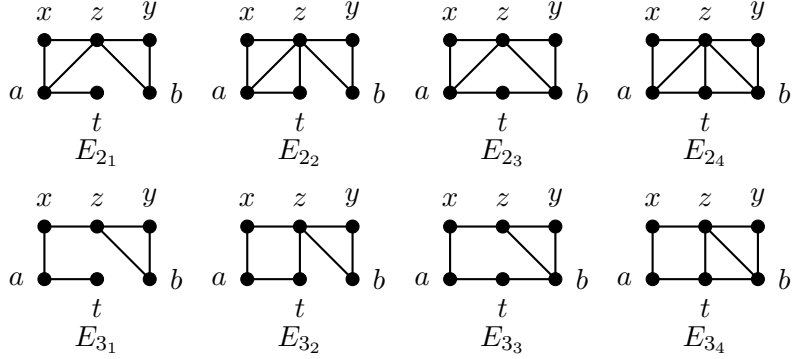
\begin{figure}[htb]
\centering
\begin{tikzpicture}[scale=0.7]
\filldraw [black]
(-1,-1) circle (3.5 pt)
(-1,0) circle (3.5 pt)
(0,0) circle (3.5 pt)
(1,-1) circle (3.5 pt)
(0,-1) circle (3.5 pt)
(1,0) circle (3.5 pt);
\node [label=below:$E_{2_1}$] (E_{2_1}) at (0,-1.5) {};
\node [label=left:$a$] (a) at (-1,-1) {};
\node [label=above:$x$] (x) at (-1,0) {};
\node [label=above:$z$] (z) at (0,0) {};
\node [label=above:$y$] (y) at (1,0) {};
\node [label=right:$b$] (b) at (1,-1) {};
\node [label=below:$t$] (t) at (0,-1) {};
\draw[thick] (-1,-1) -- (-1,0);
\draw[thick] (-1,0) -- (0,0);
\draw[thick] (0,0) -- (1,0);
\draw[thick] (1,0) -- (1,-1);
\draw[thick] (-1,-1) -- (0,0);
\draw[thick] (1,-1) -- (0,0);
\draw[thick] (-1,-1) -- (0,-1);
\end{tikzpicture}
\begin{tikzpicture}[scale=0.7]
\filldraw [black]
(-1,-1) circle (3.5 pt)
(-1,0) circle (3.5 pt)
(0,0) circle (3.5 pt)
(1,-1) circle (3.5 pt)
(0,-1) circle (3.5 pt)
(1,0) circle (3.5 pt);
\node [label=below:$E_{2_2}$] (E_{2_2}) at (0,-1.5) {};
\node [label=left:$a$] (a) at (-1,-1) {};
\node [label=above:$x$] (x) at (-1,0) {};
\node [label=above:$z$] (z) at (0,0) {};
\node [label=above:$y$] (y) at (1,0) {};
\node [label=right:$b$] (b) at (1,-1) {};
\node [label=below:$t$] (t) at (0,-1) {};
\draw[thick] (-1,-1) -- (-1,0);
\draw[thick] (-1,0) -- (0,0);
\draw[thick] (0,0) -- (1,0);
\draw[thick] (1,0) -- (1,-1);
\draw[thick] (-1,-1) -- (0,0);
\draw[thick] (1,-1) -- (0,0);
\draw[thick] (-1,-1) -- (0,-1);
\draw[thick] (0,0) -- (0,-1);
\end{tikzpicture}
\begin{tikzpicture}[scale=0.7]
\filldraw [black]
(-1,-1) circle (3.5 pt)
(-1,0) circle (3.5 pt)
(0,0) circle (3.5 pt)
(1,-1) circle (3.5 pt)
(0,-1) circle (3.5 pt)
(1,0) circle (3.5 pt);
\node [label=below:$E_{2_3}$] (E_{2_3}) at (0,-1.5) {};
\node [label=left:$a$] (a) at (-1,-1) {};
\node [label=above:$x$] (x) at (-1,0) {};
\node [label=above:$z$] (z) at (0,0) {};
\node [label=above:$y$] (y) at (1,0) {};
\node [label=right:$b$] (b) at (1,-1) {};
\node [label=below:$t$] (t) at (0,-1) {};
\draw[thick] (-1,-1) -- (-1,0);
\draw[thick] (-1,0) -- (0,0);
\draw[thick] (0,0) -- (1,0);
\draw[thick] (1,0) -- (1,-1);
\draw[thick] (-1,-1) -- (0,0);
\draw[thick] (1,-1) -- (0,0);
\draw[thick] (-1,-1) -- (0,-1);
\draw[thick] (1,-1) -- (0,-1);
\end{tikzpicture}
\begin{tikzpicture}[scale=0.7]
\filldraw [black]
(-1,-1) circle (3.5 pt)
(-1,0) circle (3.5 pt)
(0,0) circle (3.5 pt)
(1,-1) circle (3.5 pt)
(0,-1) circle (3.5 pt)
(1,0) circle (3.5 pt);
\node [label=below:$E_{2_4}$] (E_{2_4}) at (0,-1.5) {};
\node [label=left:$a$] (a) at (-1,-1) {};
\node [label=above:$x$] (x) at (-1,0) {};
\node [label=above:$z$] (z) at (0,0) {};
\node [label=above:$y$] (y) at (1,0) {};
\node [label=right:$b$] (b) at (1,-1) {};
\node [label=below:$t$] (t) at (0,-1) {};
\draw[thick] (-1,-1) -- (-1,0);
\draw[thick] (-1,0) -- (0,0);
\draw[thick] (0,0) -- (1,0);
\draw[thick] (1,0) -- (1,-1);
\draw[thick] (-1,-1) -- (0,0);
\draw[thick] (1,-1) -- (0,0);
\draw[thick] (-1,-1) -- (0,-1);
\draw[thick] (1,-1) -- (0,-1);
\draw[thick] (0,0) -- (0,-1);
\end{tikzpicture}\\
\begin{tikzpicture}[scale=0.7]
\filldraw [black]
(-1,-1) circle (3.5 pt)
(-1,0) circle (3.5 pt)
(0,0) circle (3.5 pt)
(1,-1) circle (3.5 pt)
(0,-1) circle (3.5 pt)
(1,0) circle (3.5 pt);
\node [label=below:$E_{3_1}$] (E_{3_1}) at (0,-1.5) {};
\node [label=left:$a$] (a) at (-1,-1) {};
\node [label=above:$x$] (x) at (-1,0) {};
\node [label=above:$z$] (z) at (0,0) {};
\node [label=above:$y$] (y) at (1,0) {};
\node [label=right:$b$] (b) at (1,-1) {};
\node [label=below:$t$] (t) at (0,-1) {};
\draw[thick] (-1,-1) -- (-1,0);
\draw[thick] (-1,0) -- (0,0);
\draw[thick] (0,0) -- (1,0);
\draw[thick] (1,0) -- (1,-1);
\draw[thick] (1,-1) -- (0,0);
\draw[thick] (-1,-1) -- (0,-1);
\end{tikzpicture}
\begin{tikzpicture}[scale=0.7]
\filldraw [black]
(-1,-1) circle (3.5 pt)
(-1,0) circle (3.5 pt)
(0,0) circle (3.5 pt)
(1,-1) circle (3.5 pt)
(0,-1) circle (3.5 pt)
(1,0) circle (3.5 pt);
\node [label=below:$E_{3_2}$] (E_{3_2}) at (0,-1.5) {};
\node [label=left:$a$] (a) at (-1,-1) {};
\node [label=above:$x$] (x) at (-1,0) {};
\node [label=above:$z$] (z) at (0,0) {};
\node [label=above:$y$] (y) at (1,0) {};
\node [label=right:$b$] (b) at (1,-1) {};
\node [label=below:$t$] (t) at (0,-1) {};
\draw[thick] (-1,-1) -- (-1,0);
\draw[thick] (-1,0) -- (0,0);
\draw[thick] (0,0) -- (1,0);
\draw[thick] (1,0) -- (1,-1);
\draw[thick] (1,-1) -- (0,0);
\draw[thick] (-1,-1) -- (0,-1);
\draw[thick] (0,0) -- (0,-1);
\end{tikzpicture}
\begin{tikzpicture}[scale=0.7]
\filldraw [black]
(-1,-1) circle (3.5 pt)
(-1,0) circle (3.5 pt)
(0,0) circle (3.5 pt)
(1,-1) circle (3.5 pt)
(0,-1) circle (3.5 pt)
(1,0) circle (3.5 pt);
\node [label=below:$E_{3_3}$] (E_{3_3}) at (0,-1.5) {};
\node [label=left:$a$] (a) at (-1,-1) {};
\node [label=above:$x$] (x) at (-1,0) {};
\node [label=above:$z$] (z) at (0,0) {};
\node [label=above:$y$] (y) at (1,0) {};
\node [label=right:$b$] (b) at (1,-1) {};
\node [label=below:$t$] (t) at (0,-1) {};
\draw[thick] (-1,-1) -- (-1,0);
\draw[thick] (-1,0) -- (0,0);
\draw[thick] (0,0) -- (1,0);
\draw[thick] (1,0) -- (1,-1);
\draw[thick] (1,-1) -- (0,0);
\draw[thick] (-1,-1) -- (0,-1);
\draw[thick] (1,-1) -- (0,-1);
\end{tikzpicture}
\begin{tikzpicture}[scale=0.7]
\filldraw [black]
(-1,-1) circle (3.5 pt)
(-1,0) circle (3.5 pt)
(0,0) circle (3.5 pt)
(1,-1) circle (3.5 pt)
(0,-1) circle (3.5 pt)
(1,0) circle (3.5 pt);
\node [label=below:$E_{3_4}$] (E_{3_4}) at (0,-1.5) {};
\node [label=left:$a$] (a) at (-1,-1) {};
\node [label=above:$x$] (x) at (-1,0) {};
\node [label=above:$z$] (z) at (0,0) {};
\node [label=above:$y$] (y) at (1,0) {};
\node [label=right:$b$] (b) at (1,-1) {};
\node [label=below:$t$] (t) at (0,-1) {};
\draw[thick] (-1,-1) -- (-1,0);
\draw[thick] (-1,0) -- (0,0);
\draw[thick] (0,0) -- (1,0);
\draw[thick] (1,0) -- (1,-1);
\draw[thick] (1,-1) -- (0,0);
\draw[thick] (-1,-1) -- (0,-1);
\draw[thick] (1,-1) -- (0,-1);
\draw[thick] (0,0) -- (0,-1);
\end{tikzpicture}
\caption{All induced subgraphs from $E_2$  and $E_3$ with one pendant at vertex $a$}\label{pic9}
\end{figure}
\end{proof}

%
%
%
%
%

\begin{thm}
The friendship graphs $F_1$, $F_2$ and $F_3$ are \textit{DS}.
\end{thm}
\begin{proof}
The graph $F_1$ is isomorphic to graph $K_3$, and $K_3$ is \textit{DS}. Suppose $G$ is cospectral with $F_2$, then by Theorem \ref{Theorem1} and part $(i)$ of Proposition \ref{29}, $G$ must be connected. So $G$ is connected and planar, since it does not have $K_5$ or $K_{3,3}$ as a subgraph. Therefore, by Corollary \ref{trian}, $G$ is isomorphic to $F_2$. Lastly, we prove that $F_3$ is \textit{DS}. Suppose $G$ is cospectral with $F_3$. By part $(i)$ of Proposition \ref{29}, $G$ must be connected. Also, $G$ does not have graph $K_5$ or $K_{3,3}$ as a subgraph. So by Corollary \ref{trian}, $G$ is isomorphic to $F_3$ and this completes the proof.
\end{proof}

\section{\bf  Complement of Cospectral Mate of Friendship Graph}
\vskip 0.4 true cm
In this section, we study the complement of graph $G$, where $G$ is cospectral with friendship graph $F_n$. Also,
we completely shown that the complement of friendship graph is $DS$. So, if there are some graphs that they are $\mathbb{R}$-cospectral with friendship graph, then they must be isomorphic to $F_n$. Another advantage of this section is; if we accept that $G$ is connected and induced $P_4$ free, then $G$ has a dominating vertex and it is isomorphic to friendship graph.

\begin{lem}\label{Lemma2}
Let $G$ be a cospectral graph with $F_n$ for some $n>2$. Then $\overline{G}$ is either connected or it is the disjoint union of $K_1$ and a connected graph.
\end{lem}
\begin{proof}
Since $F_n$ has $3n$ edges, $G$ has the same number of edges and so $\overline{G}$ has $n(2n+1)-3n=n(2n-2)$ edges.
If $\overline{G}$ were disconnected, then  the vertex set of $\overline{G}$ is partitioned into two parts of sizes
 $k_1$ and $k_2$ such that there is no edges between any two vertices of these two parts. So $K_{k_1,k_2}$ is a subgraph
  of $G$ and it follows that $3n\geq k_1 k_2$. Without loss of generality we may assume that $k_2\geq k_1$. Since $n>2$ and $k_1+k_2=2n+1$, it follows that $k_1=1$ and $k_2=2n$. This completes the proof.
\end{proof}
\begin{thm}
Let $G$ be a cospectral graph with $F_n$ for some $n>2$. If $\overline{G}$ is disconnected, then $G$ is isomorphic to $F_n$.
\end{thm}
\begin{proof}
By Lemma \ref{Lemma2}, $\overline{G}$ has two connected components $\overline{L}$ and $\overline{K}$, where $\overline{K}$ has only one vertex. It follows that the complement $L$ of $\overline{L}$ has $2n$ vertices and $n$ edges. If every vertex of $L$ has degree $1$ in $L$, then $L$ is the disjoint union of $n$ complete graphs $K_2$. In the latter case, $G$ will be isomorphic to $F_n$, since $G$ is the join of $L$ and $K$ which is the same graph $F_n$. Now suppose for a contradiction that there exists an isolated vertex $v$ of $L$. In the latter case, $G$ has a pendant edge at the sole vertex $w$ of $\overline{K}$. Now Theorem 2.2.1 of \cite{a11} implies that $$P_G(x)=xP_{G\setminus\{v\}}(x)-P_{G\setminus\{v,w\}}(x),$$
It follows that $P_G(0)=P_{G\setminus\{v,w\}}(0)$ and since $G\setminus\{v,w\}$ is a forest or has isolated vertex,  $|P_G(0)|$ must be $0$ or $1$. On the other hand, $|P_G(0)|$ is equal to the absolute value of the product of all eigenvalues of $G$, which is $2n$. This is a contradiction which completes the proof.
\end{proof}
\begin{lem}
Let $G$ be a graph cospectral with $F_n$ for some $n>2$. Then the eigenvalues of the complement $\overline{G}$ of $G$ are $-2$, $0$ and the roots of the following polynomial
$$x^4+(4-2n)x^3+(4-4n)x^2+(4bn^2+4cn^2-2cn+2bn-2c)x+8cn^2-4cn-4c,$$
where $b$ and $c$ are non-negative real numbers such that $b+c\leq 1$.
\end{lem}
\begin{proof}
By \cite[Proposition 2.1.3]{a11},
$$P_{\overline{G}}(x)=(-1)^{2n+1}P_G(-x-1) \left(1-(2n+1) \sum_{i=1}^4 \frac{\beta_i^2}{x+1+\mu_i}\right), \eqno{(1)}$$
where $\mu_1=\frac{1+\sqrt{8n+1}}{2}$, $\mu_2=1$, $\mu_3=-1$, $\mu_4=\frac{1-\sqrt{8n+1}}{2}$ and $\beta_1,\beta_2,\beta_3,\beta_4$ are the main angles of $G$, see \cite[page 15]{a11}. We know $$\beta_1^2+\beta_2^2+\beta_3^2+\beta_4^2=1, \eqno{(2)}$$ see \cite[page 15]{a11}, and it follows from \cite[Theorem 1.3.5]{a11} that
$$6n=(2n+1)\left(\mu_1 \beta_1^2+\beta_2^2-\beta_3^2+\mu_4 \beta_4^2\right). \eqno{(3)}$$
Now let $b:=\beta_2^2$ and $c:=\beta_3^2$. Then using  identities (2) and (3), one may simplify  $P_{\overline{G}}(x)$ given  in (1) as a product of the polynomial given in the statement of the lemma and some positive powers of  polynomials $x$ and $x+2$. This completes the proof.
\end{proof}

It is well known that the minimal non-isomorphic cospectral graphs are $G_1=C_4\cup K_1$ and $G_2=K_{1,4}$, where
$G_1=\overline{F_2}$ and $G_2$ is complete bipartite graph. So, we can see that $\overline{F_2}$ is not $DS$. The natural question is; what happen for the complement of remaining friendship graphs? We answer to this question in the next theorem.
\begin{thm}
Let $\overline{F_n}$ denote the complement of friendship graph $F_n$. Then for $n\geq 3$, $\overline{F_n}$ is $DS$.
\end{thm}
\begin{proof}
It is easy to check that the complement of friendship graph $F_n$ is $CP(n)\cup K_1$, where $CP(n)$ is cocktail party graph. The spectrum of $\overline{F_n}$ is as follows:
$$Spec(\overline{F_n})=\left\lbrace[-2]^{n-1},[0]^{n+1},[2n-2]^1\right\rbrace .$$
Let $G$ be cospectral with $\overline{F_n}$. Firstly, we prove that $G$ can not be connected. Suppose $G$ is connected. Because for $n\geq 3$, there are $\frac{1}{6}((2n-2)^3-8(n-1))$ triangles in $G$, $G$ is not bipartite and specially is not complete bipartite graph. Also in graph $G$, $(2n+1)(2n-2)$ is not eqal to the $(2n-2)^2+4(n-1)$, so by Corollary 3.2.2 of \cite{a11}, $G$ is not regular and specially is not strongly regular graph. Now by Theorem 7 of \cite{a13}, $G$ must be one of these graphs; cone over Petersen graph, the graph derived from the complement of the Fano plane, the cone over the Shrikhande graph, the cone over the lattice graph $L_2(4)$, the graph on the points and planes of $AG(3,2)$, the graph related to the lattice graph $L_2(5)$, the cones over the Chang graphs, the cone over the triangular graph $T(8)$, and the graph obtained by switching in $T(9)$ with respect to an 8-clique. But these graphs have spectrum $\left\lbrace[-2]^5,[1]^5,[5]^1\right\rbrace$, $\left\lbrace[-2]^7,[1]^6,[8]^1\right\rbrace$, $\left\lbrace[-2]^{10},[2]^6,[8]^1\right\rbrace$, $\left\lbrace[-2]^{10},[2]^6,[8]^1\right\rbrace$, $\left\lbrace[-2]^{14},[2]^7,[14]^1\right\rbrace$, $\left\lbrace[-2]^{16},[3]^7,[11]^1\right\rbrace$, $\left\lbrace[-2]^{21},[4]^7,[14]^1\right\rbrace$, $\left\lbrace[-2]^{21},[4]^7,[14]^1\right\rbrace$, and $\left\lbrace[-2]^{28},[5]^7,[21]^1\right\rbrace$, respectively. But, because of the spectrum of $G$, this is contradiction and $G$ is not connected.

Now, suppose $G$ is disconnected. By similar discussion in the proof of Theorem \ref{Theorem1}, $G$ must be the disjoint union of a connected graph $G_1$ and some isolated vertices, so $G=G_1\cup mK_1$, for some $m>0 .$ If $m>2$, then $G_1$ has $2n+1-m$ vertices and $\bigtriangleup(G_1)\leq 2n-3$, but by Theorem 3.2.1 of \cite{a11}, this is contradiction, since the index of $G_1$ is $2n-2$. If $m=2$, then $G_1$ has $2n-1$ vertices and $\bigtriangleup(G_1)\leq 2n-2$. But, the index of $G_1$ is $2n-2$ and again by Theorem 3.2.1 of \cite{a11} we must have $\bigtriangleup(G_1)=2n-2$. In this case, $G_1$ is complete graph with $2n-1$ vertices, that is contradiction. So, by the first part of proof, we have $m=1$ and $G=G_1\cup k_1$. Therefore, the spectrum of $G_1$ is $\left\lbrace[-2]^{n-1},[0]^n,[2n-2]^1\right\rbrace .$ It is well known that $CP(n)$ is $DS$ and $Spec(G_1)=Spec(CP(n)) .$ Therefore $G_1$ is isomorph to $CP(n)$ and it shows that $G=CP(n)\cup K_1=\overline{F_n}$. So, we obtain that $\overline{F_n}$ is $DS$.
\end{proof}




\bigskip
\bigskip

{\footnotesize \pn{\bf Alireza Abdollahi}\; \\ {Department of Mathematics}, {University of Isfahan,} {Isfahan 81746-73441, Iran}\\
and and School of Mathematics, Institute for Research in Fundamental Sciences (IPM), P.O.Box: 19395-5746, Tehran, Iran.\\
{\tt Email: a.abdollahi@math.ui.ac.ir}\\

{\footnotesize \pn{\bf Shahrooz Janbaz}\; \\ {Department of
Mathematics}, {University of Isfahan,} {Isfahan 81746-73441, Iran}\\
{\tt Email: shahrooz.janbaz@sci.ui.ac.ir}\\

{\footnotesize \pn{\bf Mohammad Reza Oboudi}\; \\ {Department of
Mathematics}, {University of Isfahan,} {Isfahan 81746-73441, Iran}\\
{\tt Email: mr.oboudi@sci.ui.ac.ir}\\
\end{document}